\documentclass[11pt]{article}

\usepackage{geometry}
\usepackage{graphicx}
\usepackage{amsmath, amssymb, amsfonts, amsthm, float}  
\usepackage{enumerate, color, framed, float, multirow}
\usepackage{comment, longtable, caption, subcaption, appendix}
\usepackage[sort,longnamesfirst]{natbib}

\binoppenalty=\maxdimen
\relpenalty=\maxdimen

\allowdisplaybreaks

\newcommand{\ds}{\displaystyle}
\newcommand{\E}{\text{E}}

\newcommand{\X}{\mathsf{X}}
\newcommand{\B}{\mathcal{B}}
\newcommand{\real}{{\mathbb R}}

\newcommand\numberthis{\addtocounter{equation}{1}\tag{\theequation}}

\newtheorem{theorem}{Theorem}

\newtheorem{lemma}{Lemma}

\theoremstyle{remark}

\newtheorem{remark}{Remark}

\title{Geometric Ergodicity of Gibbs Samplers in Bayesian Penalized Regression Models}
\author{Dootika Vats\\
Department of Statistics\\
University of Warwick\\
Coventry, UK\\
\texttt{D.Vats@warwick.ac.uk}\\}
\date{\today}

\begin{document}
\maketitle

\begin{abstract}
We consider three Bayesian penalized regression models and show that the respective deterministic scan Gibbs samplers are geometrically ergodic regardless of the dimension of the regression problem. We prove geometric ergodicity of the Gibbs samplers for the Bayesian fused lasso, the Bayesian group lasso, and the Bayesian sparse group lasso. Geometric ergodicity along with a moment condition results in the existence of a Markov chain central limit theorem for Monte Carlo averages and ensures reliable output analysis. Our results of geometric ergodicity allow us to also provide default starting values for the Gibbs samplers.
\end{abstract}

\section{Introduction} 
\label{sec:introduction}
Let $y \in \real^n$ be the observed realization of the response $Y$, $X$ be the $n \times p$ model matrix, and $\beta \in \real^p$ be  the regression coefficient vector. The goal, generally, is to identify important predictors amongst the $p$ covariates and estimate the corresponding coefficients in $\beta$. However, in many problems, like 
genetics, image processing, chemometrics, economics,  the number of covariates, $p$ can be much larger than $n$, making it difficult to use classical regression techniques. Bayesian and frequentist penalization methods have been found to be very useful in such situations. Consider the Bayesian regression model of the form
\begin{align*}
	Y \mid \beta, \sigma^2 \;\; & \sim \;\; N_n(X \beta, \sigma^2 I_n)\\
	\beta \mid \eta, \sigma^2 \;\; & \sim \;\; N_p(0, \sigma^2 \Sigma_{\eta})\\
	\eta \;\; & \sim \;\; p(\eta)\\
	\sigma^2 \;\; & \sim \;\; \text{Inverse-Gamma}(\alpha, \xi)\,, \numberthis \label{eq:model_general}
\end{align*}
where $\alpha, \xi \geq 0$ are assumed known, $\Sigma_{\eta}$ is a $p \times p$ covariance matrix determined by $\eta \in \real^s_+$, and $p(\eta)$ is a proper prior on $\eta$. Many Bayesian penalized regression and variable selection models can be presented in this framework (see for example \cite{guan:step:2011, kyu:gill:gho:cas:2010, park:cas:2008, yang:jor:wain:2016}).   The resulting posteriors are often intractable and Markov chain Monte Carlo (MCMC) is used to estimate model parameters.

Consider the  Bayesian fused lasso, the Bayesian group lasso \cite{kyu:gill:gho:cas:2010}, and the Bayesian sparse group lasso \cite{xu:gho:2015}, all three of which belong to the family of models in \eqref{eq:model_general}. These models have been used in a variety of problems. The Bayesian group lasso and the Bayesian sparse group lasso find use in medical research \cite{fan:wang:peng:2017, gu:yin:lee:2013, nath:green:lesp:2016, ram:fuc:wild:2010}.  The Bayesian fused lasso has been used in breast cancer research \cite{zhang2014bayesian}. Given the use of these models in medical research, reliable inference is essential. 

Reliable estimation from MCMC output rests heavily on the rate of convergence of the Markov chain. In particular, a geometric rate of convergence lets users appeal to the Markov chain central limit theorem (CLT), allowing for the  estimation of Monte Carlo error in posterior estimates and consistent estimation of effective sample size. We show that the MCMC samplers used in the three models converge to their respective stationary distribution at a geometric rate. That is, we show that the Gibbs samplers are \textit{geometrically ergodic} (formal definitions are in Section \ref{sec:markov_chain_background}).

In the models we study, the full conditionals for $\beta$, $\eta$ and $\sigma^2$ are available in closed form so that it is straightforward to draw samples from $f(\beta \mid \eta, \sigma^2, y), f(\eta \mid \beta, \sigma^2, y)$, and $f(\sigma^2 \mid \beta, \eta, y)$. As a consequence, a three variable deterministic scan Gibbs sampler is implemented to draw approximate samples from the intractable posterior distribution and inference is  done using sample statistics.  The quality of estimation is affected not only by the size of the Monte Carlo sample, but also by the rate of convergence of the Gibbs sampler.  We show that all three Gibbs samplers converge to their respective stationary distribution at a geometric rate under reasonable conditions. Specifically, we only require the number of observations, $n$, to be larger than three and require no assumptions on the number of covariates, $p$ or the model matrix $X$. This geometric rate of convergence allows for reliable estimation of posterior quantities in the following way.

Let $F$ denote the posterior distribution of $(\beta, \eta,\sigma^2)$ obtained from \eqref{eq:model_general}, defined on the space $\X = \real^p \times \real^s_+ \times \real_+$ and let $f(\beta, \eta, \sigma^{2} \mid y)$ be the associated density. 
Let $g:\X \to \real^d$ be an $F$-integrable function, then interest is in estimating
\[\theta := \int_{\X}g(\beta, \eta, \sigma^2) \; f(\beta, \eta, \sigma^2 \mid y) d\beta\, d\eta\, d\sigma^2 < \infty\,. \]
Typically $\theta$ represents means, variance or quantiles of the posterior distribution. For $t = 0, 1, 2, \dots$, let  $(\beta^{(t)}, \eta^{(t)}, \sigma^{2(t)})$ be the samples obtained using a Harris ergodic Gibbs sampler. Then, with probability 1, for every $(\beta^{(0)}, \eta^{(0)}, \sigma^{2(0)}) \in \X$
\[\theta_N := \dfrac{1}{N} \ds \sum_{t=0}^{N-1} g(\beta^{(t)}, \eta^{(t)}, \sigma^{2(t)}) \to \theta \quad \text{ as $N \to \infty$}\,. \]
However, in finite samples there is typically a non-zero Monte Carlo error $\theta_N - \theta$ and an approximate sampling distribution of this error maybe available via a Markov chain CLT. Let $\| \cdot \|$ denote the Euclidean norm. If the deterministic scan Gibbs sampler is geometrically ergodic and 
\[\int_{\X} \left\|g(\beta, \eta, \sigma^2) \right\|^{2 + \delta} f(\beta , \eta, \sigma^2 \mid y) d\beta \; d\eta\; d\sigma^2 < \infty\,, \]
then a Markov chain CLT holds as below:
\begin{equation}
\label{eq:clt}
	\sqrt{n}(\theta_N - \theta) \overset{d}{\to} N_d(0, \Sigma)\quad  \text{as } N \to \infty\,,
\end{equation}
where $\Sigma$ is the $d \times d$ asymptotic covariance matrix that is difficult to calculate due to the serial correlation in the Markov chain. However, if the process is geometrically ergodic, then \cite{vats:fleg:jones:2015b} and \cite{vats:fleg:jones:2015} provide strongly consistent estimators of $\Sigma$. This leads to the construction of asymptotically valid confidence ellipsoids around $\theta_N$ and consistent estimation of effective sample size \cite{vats:fleg:jones:2015b}. Under the assumption of geometric ergodicity, the diagonals of $\Sigma$ were estimated by \cite{fleg:gong:2015}, \cite{fleg:jone:2010}, \cite{gong:fleg:2015}, \cite{hobe:jone:pres:rose:2002}, and \cite{jone:hara:caff:neat:2006} leading to reliable univariate analysis of MCMC output. For estimating quantiles, \cite{doss:fleg:jone:neat:2014} show that geometric ergodicity leads to  strongly consistent estimators of the Monte Carlo error.

There has been a considerable amount of work done in establishing geometric ergodicity of Gibbs samplers; many of which are two variable Gibbs samplers. Two variable Gibbs samplers are special because the marginal process for each  variable is a Markov chain with the same rate of convergence as the joint chain. Thus, it is sufficient to study the marginal chains to ascertain the properties of the joint chain. Higher variable Gibbs samplers do not benefit from this property and thus studying their rate of convergence is often more challenging.  Geometric ergodicity of the three variable Gibbs samplers in the Bayesian lasso and the Bayesian elastic net were shown by \cite{khare:hobe:2013} and \cite{roy:chakra:2016}, respectively; \cite{pal:khar:2014} proved geometric ergodicity of the three variable Gibbs sampler for the normal-gamma model of \cite{grif:brown:2010}; \cite{khar:hob:2012} demonstrated geometric ergodicity of the three variable Gibbs sampler in Bayesian quantile regression, and \cite{doss:hobe:2010} and \cite{jone:hobe:2004} demonstrated geometric ergodicity of the three variable Gibbs sampler in hierarchical random effects models. Recently, \cite{john:jone:2015} established geometric ergodicity of a four variable random scan Gibbs sampler for a hierarchical random effects model.

The rest of the paper is organized as follows.  In Section \ref{sec:markov_chain_background} we present important definitions and some relevant Markov chain background. In Section \ref{sec:bfl}, Section \ref{sec:bgl}, and Section \ref{sec:bsgl} we present the models and main results for the Bayesian fused lasso, Bayesian group lasso, and the Bayesian sparse group lasso. We finish with a discussion in Section \ref{sec:discussion}. All proofs are deferred to the appendices.

\section{Markov Chain Background} 
\label{sec:markov_chain_background}

Recall that $F$ denotes the posterior distribution of $(\beta, \eta,\sigma^2)$ obtained from \eqref{eq:model_general} and $f(\beta, \eta, \sigma^{2} \mid y)$ is the associated density. Also recall  that $\X = \real^p \times \real^s_+ \times \real_+$ is the support of the posterior and let $\B(\X)$ denote the Borel $\sigma$-algebra. Let $f(\beta \mid \eta, \sigma^2, y)$ be the density of the full conditional distribution of $\beta$ and similarly denote the densities of the conditional distributions of $\eta$ and $\sigma^2$ with $f(\eta \mid \beta, \sigma^2, y)$ and $f(\sigma^2 \mid \beta, \eta, y)$, respectively.  Let $(\beta^{(0)}, \eta^{(0)}, \sigma^{2(0)})$ be the starting value for the Gibbs sampler and define the Markov chain transition density (MTD) for the deterministic scan Gibbs sampler as
\begin{align*}
k\left( (\beta^{(1)}, \eta^{(1)}, \sigma^{2(1)}) \mid (\beta^{(0)}, \eta^{(0)}, \sigma^{2(0)}) \right) & = f(\beta^{(1)} \mid \eta^{(1)}, \sigma^{2(1)}, y)\\
& \quad \times f(\eta^{(1)} \mid \beta^{(0)}, \sigma^{2(1)}, y)\\
& \quad \times f(\sigma^{2(1)} \mid \beta^{(0)}, \eta^{(0)}, y)\,.	
\end{align*}
Then, the one-step transition kernel $P:\X \times \B(\X) \to [0,1]$ is such that for any $A \in \B(\X)$,
\begin{align*}
P\left((\beta^{(0)}, \eta^{(0)}, \sigma^{2(0)}) ,A   \right) & = \Pr \left((\beta^{(1)}, \eta^{(1)}, \sigma^{2(1)}) \in A \mid (\beta^{(0)}, \eta^{(0)}, \sigma^{2(0)})  \right)\\
& = \int_A  k\left( (\beta^{(1)}, \eta^{(1)}, \sigma^{2(1)}) \mid (\beta^{(0)}, \eta^{(0)}, \sigma^{2(0)}) \right)  d\beta^{(1)}\; d\eta^{(1)} \; d\sigma^{2(1)}\,.	
\end{align*}
Similarly, the $t$-step Markov chain transition kernel for the deterministic scan Gibbs sampler is $P^t:\X \times \B(\X) \to [0,1]$ such that for all $A \in \B(\X)$,
\[P^t\left((\beta^{(0)}, \eta^{(0)}, \sigma^{2(0)}), A\right) = \Pr\left((\beta^{(t)}, \eta^{(t)}, \sigma^{2(t)})  \in A \mid (\beta^{(0)}, \eta^{(0)}, \sigma^{2(0)})  \right)\,. \]

Let $\| \cdot \|_{TV}$ denote total variation norm. If the Markov chain is aperiodic, irreducible, and Harris recurrent (see \cite{meyn:twee:2009} for definitions), then for all $(\beta^{(0)}, \eta^{(0)}, \sigma^{2(0)}) \in \X$
\[\Big\|P^t\left((\beta^{(0)}, \eta^{(0)}, \sigma^{2(0)}),  \, \cdot \, \right) - F (\cdot) \Big\|_{TV} \to 0 \text{ as } t \to \infty\,. \] 
However, convergence of the transition kernel to the invariant distribution is not sufficient to ensure reliable inference and a geometric rate of convergence is often required. The Gibbs sampler is \textit{geometrically ergodic} if there exists a function $M:\X \to [0,\infty)$ and $0\leq \rho< 1$ such that for all $(\beta^{(0)}, \eta^{(0)}, \sigma^{2(0)}) \in \X$,
\begin{equation}
\label{eq:geom_ergodicity_definition}
	\Big \|P^t\left((\beta^{(0)}, \eta^{(0)}, \sigma^{2(0)}), \, \cdot \, \right) - F (\cdot) \Big\|_{TV} \leq M\left((\beta^{(0)}, \eta^{(0)}, \sigma^{2(0)}) \right) \rho^t\,. 
\end{equation}
Since $\rho < 1$, the upper bound in \eqref{eq:geom_ergodicity_definition} decreases at a geometric rate as a function of $t$. We will show that the three Gibbs samplers are geometrically ergodic by establishing a \textit{drift condition} and an associated \textit{minorization condition}. In effect, we will determine $M$ up to a proportionality constant and minimize this quantity to arrive at default starting values for the Gibbs samplers. Our results can also be used to obtain quantitative upper bounds for \eqref{eq:geom_ergodicity_definition} using the results of \cite{rose:1995a}; we do not explore that here.

Geometric ergodicity is often demonstrated by establishing a drift condition and an associated minorization condition. A drift condition is said to hold if there exists a function $V: \X \to [0, \infty)$, and constants $0 < \phi < 1$ and $ L < \infty$ such that for all $(\beta_0, \eta_0, \sigma^2_0) \in \X$
\begin{equation}
\label{eq:bfl_drift_condition}
	\E\left[ V(\beta, \eta,\sigma^2) \mid \beta_0, \eta_0, \sigma^2_0\right] \leq \phi V(\beta_0, \eta_0, \sigma^2_0)+ L\,.
\end{equation}
In \eqref{eq:bfl_drift_condition}, the expectation is with respect to the MTD for the Gibbs sampler.

Consider for $d > 0$, the set $C_d = \left\{(\beta, \eta, \sigma^2): V(\beta, \eta, \sigma^2) \leq d \right\}$. A minorization condition holds if there exists an $\epsilon > 0$ and a distribution $Q$ such that for all $(\beta_0, \eta_0, \sigma^2_0) \in C_d$
\begin{equation}
\label{eq:bfl_minorization}
P\left( (\beta_0, \eta_0, \sigma^2_0), \, \cdot \, \right)  \geq \epsilon \, Q(\cdot)\,.	
\end{equation}
It is well know that both \eqref{eq:bfl_drift_condition} and \eqref{eq:bfl_minorization} together imply geometric ergodicity (see \cite{jone:hobe:2001} and \cite{meyn:twee:2009}). The \textit{drift rate} $\phi$ determines how fast the Markov chain drifts back to the \textit{small set} $C_d$. A drift rate close to one signifies slower convergence and a smaller value indicates faster convergence. See \cite{jone:hobe:2001} for a heuristic explanation.

When a drift condition holds, \cite{meyn:twee:2009}, \cite{robe:rose:1997c}, and \cite[Fact 10]{robe:rose:2004} explain that the function $M$ is proportional to the \textit{drift function} $V$ up to an unknown constant. Thus, minimizing $V$ over the state space leads to the tightest bound in \eqref{eq:geom_ergodicity_definition} for our choice of $V$. This will lead us to default starting values for the three Gibbs sampler.

\section{Bayesian Fused Lasso} 
\label{sec:bfl}

Recall that $y \in \real^n$ is the observed realization of the response $Y$, $X$ is the $n \times p$ model matrix, and $\beta \in \real^p$ is the regression coefficient vector. \citet{tib:saun:ross:2005} proposed the fused lasso in an effort account for ordering in the predictors. In addition to penalizing the $L_1$ norm of the coefficients, the fused lasso also penalizes pairwise differences. That is, for tuning parameters $\lambda_1, \lambda_2 > 0$, the fused lasso estimate is,
\begin{equation}
	\label{eq:freq_fused_lasso}
	\hat{\beta}_{\text{fused}} = \arg\max_{\beta} \|y - X\beta\|^2 + \lambda_1\sum_{j=1}^{p} |\beta_j| + \lambda_2 \sum_{j=1}^{p-1} |\beta_{j+1} - \beta_{j}|\,.
\end{equation}
A Bayesian formulation of the fused lasso requires a prior on $\beta$ so that the resulting posterior mode is the $\hat{\beta}_{\text{fused}}$. \citet{kyu:gill:gho:cas:2010} present the following Bayesian formulation of the fused lasso. Let 
\begin{align*}
Y \mid \beta, \sigma^2, \tau^2  & \sim  N_n(X\beta, \sigma^2 I_n)\\
\beta \mid \tau^2, w^2, \sigma^2 & \sim N_p(0, \sigma^2 \,\Sigma_{\tau,w})\\
	\tau^2_i &\overset{\text{ind}}{\sim}  \dfrac{\lambda_1^2}{2} e^{-\lambda_1 \tau^2_i/2} d\tau_i^2    \quad \text{ for } \tau^2_i > 0 , i = 1, \dots, p\numberthis \label{eq:kyung_prior}\\
	w^2_i &\overset{\text{ind}}{\sim} \dfrac{\lambda_2^2}{2} e^{-\lambda_2 w^2_i/2} dw_i^2 \quad \text{ for } w_i^2 > 0, i = 1, \dots, p-1\\ 
\sigma^2 & \sim \text{Inverse-Gamma}(\alpha, \xi)\,,
\end{align*}
where $\alpha, \xi \geq 0$ are known, $\lambda_1, \lambda_2 > 0$ are fixed, and $\Sigma_{\tau, w}$ is such that $\Sigma^{-1}_{\tau, w}$ is a tridiagonal matrix with main diagonals
\begin{align*}
\left(\dfrac{1}{\tau^2_1} + \dfrac{1}{w_1^2}\right), \, \, \left(\dfrac{1}{\tau^2_i} + \dfrac{1}{w^2_{i-1}} + \dfrac{1}{w^2_{i}}\right) \; \text{ for } i = 2, \dots, p-1, \text{ and } \left(\dfrac{1}{\tau^2_p} + \dfrac{1}{w_{p-1}^2} \right)\,,
\end{align*}
and off diagonals $\{-1/w_i^2 :i = 1, \dots, p\}$.
%
Specifically, $\Sigma_{\tau, w}^{-1}$ takes the following form,
\begin{equation}
\label{eq:sigma_beta}
\Sigma^{-1}_{\tau, w} = \left[\begin{array}{ccccc} \frac{1}{\tau^2_1} + \frac{1}{w^2_1} & -\frac{1}{w^2_1} & 0 & \dots & 0 \\ -\frac{1}{w^2_1} & \frac{1}{\tau^2_2} + \frac{1}{w^2_1} + \frac{1}{w^2_2} & -\frac{1}{w^2_2} & \dots  & 0 \\ 0 & -\frac{1}{w_2^2} & \frac{1}{\tau^2_3} + \frac{1}{w^2_2} + \frac{1}{w^2_3} & \dots &  0\\ \dots & \dots & \dots & \ddots & \dots \\ 0 & 0 & \dots & \frac{1}{\tau^2_{p-1}} + \frac{1}{w^2_{p-2}} + \frac{1}{w^2_{p-1}}  & -\frac{1}{w^2_{p-1}}  \\  0 & 0 & \dots & -\frac{1}{w^2_{p-1}} & \frac{1}{\tau^2_p} + \frac{1}{w^2_{p-1}} \end{array}   \right]\,.
\end{equation}
Let $\tau^2 = (\tau^2_1, \dots, \tau^2_p)$ and $w^2 = (w_1^2, \dots, w_{p-1}^2)$. \citet{kyu:gill:gho:cas:2010} state that the priors in \eqref{eq:kyung_prior} lead to the following marginal prior on $\beta$ given $\sigma^2$.
\begin{equation}\label{eq:laplace_fused}
	\pi(\beta \mid \sigma^2) \propto \exp \left(-\dfrac{\lambda_1}{\sigma} \ds \sum_{j=1}^{p} | \beta_j | - \dfrac{\lambda_2}{\sigma} \ds \sum_{j=1}^{p-1} |\beta_{j+1} - \beta_j| \right)\,.
\end{equation}
However, this is not the case and in particular, the independent exponential priors on $\tau^2$ and $w^2$  do not lead to the marginal prior in \eqref{eq:laplace_fused}. Instead, our proposed prior is
\begin{equation}
\label{eq:bfl_correct_prior}
\pi(\tau^2, w^2) \propto \det\left(\Sigma_{\tau,w} \right)^{1/2} \left( \prod_{i=1}^{p} \left(\tau^2_i \right)^{-1/2} e^{-\lambda_1 \tau^2_i/2} \right) \left(\prod_{i=1}^{p-1} \left(w^2_i \right)^{-1/2}e^{-\lambda_2 w^2_i/2} \right)\,.
\end{equation}
In Appendix \ref{sub:propriety_of_prior}, we show that the prior on $(\tau^2, w^2)$ in \eqref{eq:bfl_correct_prior} is proper and in Appendix \ref{sub:validity_of_the_prior} we demonstrate that the marginal prior on $\beta$ given $\sigma^2$ is the appropriate prior in \eqref{eq:laplace_fused}. Thus, our model formulation is a valid Bayesian fused lasso model. 


\subsection{Gibbs Sampler for the Bayesian Fused Lasso} 
\label{sub:gibbs_sampler_for_the_bfl}
The resulting full conditionals from the model in \eqref{eq:kyung_prior} with prior \eqref{eq:bfl_correct_prior} are,
\begin{align*}
	\beta \mid \sigma^2, \tau^2, w^2, y & \; \; \sim \; \; N_p \left( (X^TX + \Sigma_{\tau,w}^{-1})^{-1} X^T y , \sigma^2 (X^TX + \Sigma_{\tau, w}^{-1} )^{-1}   \right)\\
	\dfrac{1}{\tau^2_{i}} \mid \beta, \sigma^2, y &  \; \; \overset{\text{ind}}{\sim}  \; \; \text{Inverse-Gaussian} \left( \sqrt{\dfrac{\lambda_1^2 \sigma^2}{\beta^2_i}}, \lambda_1^2 \right), \text{ for all $i = 1, \dots, p$}\\
	\dfrac{1}{w_i^2} \mid \beta,\sigma^2, y &  \; \; \overset{\text{ind}}{\sim} \text{Inverse-Gaussian} \; \;  \left( \sqrt{\dfrac{\lambda_2^2 \sigma^2}{ (\beta_{i+1} - \beta_{i})^2}}, \lambda_2^2 \right), \text{ for all $i = 1, \dots, p-1$}\\
	\sigma^2 \mid \beta, \tau^2, w^2, y &  \; \; \sim  \; \; \text{Inverse-Gamma} \left(\dfrac{n + p + 2\alpha}{2}, \dfrac{(y - X\beta)^T(y - X\beta) + \beta^T \Sigma^{-1}_{\tau, w}\beta + 2\xi }{2}   \right)\,. \numberthis \label{eq:bfl_full_cond}
\end{align*}
Here the Inverse-Gaussian$(a,b)$ density is $f(x) \propto x^{-3/2} \exp(-b(x-a)^2/2a^2x)$ and the density of an Inverse-Gamma$(a,b)$ distribution is $f(x) \propto x^{-a-1} \exp(-b/x)$.
Notice that the full conditionals for $\tau^2$ and $w^2$ are independent and thus can be updated in one block. This reduces the four variable Gibbs sampler to a three variable Gibbs sampler. If $(\beta_{(t)}, \tau^2_{(t)}, w^2_{(t)}, \sigma^2_{(t)})$ is the current state of the Gibbs sampler the $(t+1)$th state is obtained as follows.

\noindent\rule{15.1cm}{0.5pt}
\begin{enumerate}
	\item Draw $\sigma^{2}_{(n+1)}$ from $f(\sigma^2 \mid \beta_{(n)}, \tau^2_{(n)}, w^2_{(n)}, y)$.
	\item Draw $\left(1/\tau^2_{(n+1)}, 1/w^2_{(n+1)} \right)$ from $f(1/\tau^2 \mid \beta_{(n)}, \sigma^2_{(n+1)}, y) \;f(1/w^2 \mid \beta_{(n)}, \sigma^2_{(n+1)}, y)$.
	\item Draw $\beta_{(n+1)}$ from $f(\beta \mid \tau^2_{(n+1)}, w^2_{(n+1)}, \sigma^2_{(n+1)}, y)$.
\end{enumerate}
\noindent\rule{15.1cm}{0.5pt}\\
This three variable deterministic scan Gibbs sampler has MTD,
\begin{align*}
	& k_{BFL}(\beta, \tau^2, w^2, \sigma^2 \mid \beta_0, \tau^2_0, w_0^2, \sigma^2_0)\\
	& = f(\beta \mid \tau^2, w^2,  \sigma^2, y) \;f(\tau^2, w^2 \mid \beta_0, \sigma^2, y) f(\sigma^2 \mid \beta_0, \tau^2_0, w^2_0, y)\label{eq:mtd} \numberthis\,.
\end{align*}
First we note that the full conditional distribution of $1/\tau^2_i$ is an Inverse-Gaussian with mean parameter $\sqrt{\lambda_1^2 \sigma^2/ \beta^2_i}$. If the starting value for any $\beta_{i}$ is zero, this Inverse-Gaussian is still well defined as it is an Inverse-Gamma distribution with shape parameter  $1/2$ and rate parameter $\lambda_1^2/2$. The same is true for the full conditional of $1/w_i^2$. Thus, the MTD is strictly positive and well defined which implies the Markov chain is aperiodic, irreducible almost everywhere, and Harris recurrent. 


We define the drift function $V_{BFL}:\real^p \times \real_+^p \times \real_+^{p-1} \times \real_+ \to [0, \infty)$ as
\begin{equation}
\label{eq:bfl_drift}
	V_{BFL}(\beta, \tau^2, w^2, \sigma^2) = (y - X\beta)^T(y - X\beta) + \beta^T \Sigma^{-1}_{\tau, w} \beta + \dfrac{\lambda_1^2}{4} \ds \sum_{i=1}^{p} \tau^{2}_{i} + \dfrac{\lambda_2^2}{4} \ds \sum_{i=1}^{p-1} w_i^2\,.
\end{equation}
The following theorem is proved by establishing \eqref{eq:bfl_drift_condition} and \eqref{eq:bfl_minorization} for the drift function $V_{BFL}$.
\begin{theorem}
\label{thm:ge_bfl}
If $n \geq 3$, the three variable Gibbs sampler for the Bayesian fused lasso is geometrically ergodic.
\end{theorem}
\begin{proof}
	See Appendix \ref{sec:bfl_proof}.
\end{proof}

\begin{remark}\label{rem:bfl_drift_analyse}
In Appendix \ref{sub:bfl_drift_condition}, we arrive at the drift rate
\[\phi_{BFL} = \max \left\{ \dfrac{p}{n + p + 2\alpha - 2}, \dfrac{1}{2} \right\}\,.
\] Thus, $\phi_{BFL}$  is no better than 1/2 and as $p$ increases, the drift rate approaches one. Thus, convergence may be slower for problems with large $p$.
\end{remark}

\begin{remark}\label{rem:bfl_starting_value}
Minimizing $V_{BFL}$ yields default starting value of $\beta_0$ being the frequentist fused lasso estimate, $\tau^2_{0,i} = 2|\beta_{0,i}|/\lambda_1$ and $w^2_{0,i} = 2|\beta_{0, i+1} - \beta_{0,i}|/\lambda_2$. See Appendix \ref{sub:bfl_starting_values} for details.
\end{remark}




\section{Bayesian Group Lasso} 
\label{sec:bgl}
Knowledge of correlation among predictors is ignored by the usual lasso. The group lasso of  \cite{yuan:lin:2006}  imposes sparsity across grouped predictors. For a fixed $K$, partition $\beta$ in $K$ groups of size $m_1, m_2, \dots, m_K$; the groups being denoted by $\beta_{G_1}, \beta_{G_2}, \dots, \beta_{G_K}$. Let $X_{G_k}$ denote the matrix of predictors for group $k$. The group lasso estimate for tuning parameter $\lambda > 0$ is,
\begin{equation}
	\label{eq:fre_group_lasso}
	\hat{\beta}_{\text{group}} = \arg \max_{\beta} \Big\|y - \sum_{k=1}^{K} X_{G_k} \beta_{G_k} \Big\|^2 + \lambda\sum_{k=1}^{K} \|\beta_{G_k}\|\,.
\end{equation}
\citet{kyu:gill:gho:cas:2010} present the following Bayesian analog of the group lasso. Let
\begin{align*}
Y \mid \beta, \sigma^2 \; \; & \sim  \; \;  N_n(X\beta, \sigma^2 I_n) \\
\beta_{G_k} \mid \sigma^2, \tau^2_k  \; \; & \overset{\text{ind}}{\sim}  \; \;  N_{m_k}(0, \sigma^2 \tau^2_{k}I_{m_k}) \quad k = 1, \dots, K \numberthis \label{eq:bsl_model} \\
\tau^2_{k} \; \; & \overset{\text{ind}}{\sim}  \; \; \text{Gamma} \left( \dfrac{m_k + 1}{2}, \dfrac{\lambda^2}{2} \right) \quad k = 1, \dots, K\\
\sigma^2  \; \; & \sim  \; \; \text{Inverse-Gamma}(\alpha, \xi)\,,
\end{align*}
where $\lambda > 0$ is fixed, $\alpha, \xi \geq 0$ are known, and the density of a Gamma$(a,b)$ is $f(x)$ $\propto$ $x^{a - 1} e^{-bx}$.

\subsection{Gibbs Sampler for Bayesian Group Lasso} 
\label{sub:gibbs_sampler_for_bgl}

Let $\tau^2 = (\tau^2_1, \tau^2_2, \dots, \tau^2_K)$. Define 
\[D_{\tau}  = \text{diag}(\underbrace{\tau^2_1, \dots, \tau^2_{1}}_{m_1}, \underbrace{\tau^2_2, \dots, \tau^2_{2}}_{m_2}, \dots, \underbrace{\tau^2_{K}, \dots, \tau^2_{K}}_{m_K} )\,. \]
The Bayesian group lasso in \eqref{eq:bsl_model} leads to the following full conditionals for $\beta, \tau^2$ and $\sigma^2$:
\begin{align*}
	\beta \mid \sigma^2, \tau^2, y & \; \; \sim  \; \; N_p \left( (X^TX + D_{\tau}^{-1})^{-1} X^T y , \sigma^2 (X^TX + D^{-1}_{\tau} )^{-1}  \right)\\
	\dfrac{1}{\tau^2_{k}} \mid \beta, \sigma^2, y &  \; \; \overset{\text{ind}}{\sim} \; \; \text{Inverse-Gaussian} \left( \sqrt{\dfrac{\lambda^2 \sigma^2}{\beta^T_{G_k} \beta_{G_k} }}, \lambda^2 \right), \text{ for } k = 1, \dots, K \numberthis \label{eq:full_cond}\\
	\sigma^2 \mid \beta, \tau^2, y &  \; \; \sim \; \; \text{Inverse-Gamma} \left(\dfrac{n + p + 2\alpha}{2}, \dfrac{(y - X\beta)^T(y - X\beta) + \beta^TD^{-1}_{\tau}\beta + 2\xi }{2}   \right)\,. 
\end{align*}
These full conditionals lead to a three variable Gibbs sampler where the variables are $\beta, \tau^2$, and $\sigma^2$.

\begin{remark}\label{rem:kyung_error}
 \cite{kyu:gill:gho:cas:2010} propose a $K+2$ variable Gibbs sampler where the variables are $\beta_{G_1}, \beta_{G_2}, \dots, \beta_{G_K}, \tau^2$, and $\sigma^2$. For this sampler, the full conditionals for $\sigma^2$ and $\tau^2$ are the same as above, but the full conditional for each $\beta_{G_k}$ is
 \begin{align*}
 &\beta_{G_k} \mid \beta_{-G_k}, \sigma^2, \tau^2, y \\
 & \sim N_{m_k} \left( \left(X_{G_k}^TX_{G_k} + \tau^{-2}_{k} I_{m_k} \right)^{-1} X^T_{G_k} \left(y - \sum_{k'\ne k}X_{G_k'} \beta_{G_{k'}}\right), \sigma^2 \left(X_{G_k}^TX_{G_k} + \tau^{-2}_{k} I_{m_k}\right)^{-1}  \right)\,. 	
 \end{align*}
\cite{kyu:gill:gho:cas:2010} had an error in their full conditional where they had 
\[\left(y - \dfrac{1}{2}\sum_{k'\ne k}X_{G_k'} \beta_{G_{k'}}\right) \text{ instead of } \left(y - \sum_{k'\ne k}X_{G_k'} \beta_{G_{k'}}\right)\,.\]
The motivation for using the $K+2$ sampler is to avoid the $p \times p$ matrix inversion of $(X^TX + D_{\tau}^{-1})$, and instead do $K$ matrix inversions each of size $m_k \times m_k$. This reduces the computational cost from $O(p^3)$ to $O(\sum_{k=1}^{K} m_k^3)$. Such a technique was also discussed in \cite{ishw:rao:2005}. However, it is known that a blocked Gibbs sampler mixes as well as or better than a full Gibbs sampler (see \cite{liu:wong:kong:1994}). In addition, \cite{bhatt:chak:mall:2015} recently proposed a linear time sampling algorithm to sample from high-dimensional normal distributions of the form in \eqref{eq:full_cond}. Using their method, the computational cost of drawing from the full conditional of $\beta$ is $O(n^2p)$, and thus the $K + 2$ variable Gibbs sampler is not required.
\end{remark}

We will study the rate of convergence of the three variable Gibbs sampler. If $(\beta_{(t)}, \tau^{2}_{(t)}, \sigma^{2}_{(t)})$ is the current state of the Gibbs sampler, the $(t+1)$th state is obtained as follows.

\noindent\rule{15.1cm}{0.5pt}
\begin{enumerate}
	\item Draw $\sigma^{2}_{(n+1)}$ from $f(\sigma^2 \mid \beta_{(n)}, \tau^2_{(n)}, y)$.
	\item Draw $1/\tau^2_{(n+1)}$ from  $f(1/\tau^2 \mid \beta_{(n)}, \sigma^2_{(n+1)}, y)$.
	\item Draw $\beta_{(n+1)}$ from  $f(\beta \mid \tau^2_{(n+1)}, \sigma^2_{(n+1)}, y)$.
\end{enumerate}
\noindent\rule{15.1cm}{0.5pt}\\
The MTD for the above three variable deterministic scan Gibbs sampler is
\begin{equation}
\label{eq:mtd}
	k_{BGL}(\beta, \tau^2, \sigma^2 \mid \beta_0, \tau^2_0, \sigma^2_0) = f(\beta \mid \tau^2, \sigma^2, y) \;f(\tau^2 \mid \beta_0, \sigma^2, y) \; f(\sigma^2 \mid \beta_0, \tau^2_0, y)\,.
\end{equation}
%
As in the Bayesian fused lasso, the MTD is well defined and strictly positive leading to an aperiodic, irreducible almost everywhere, and Harris recurrent Markov chain. 

Define the drift function $V_{BGL}:\real^p \times \real_+^K \times \real_+ \to [0, \infty)$ as
\begin{equation}
\label{eq:drfit_bgl}
	V_{BGL}(\beta, \tau^2, \sigma^2) = (y - X\beta)^T(y - X\beta) + \beta^T D_{\tau}^{-1} \beta + \dfrac{\lambda^2}{4} \ds \sum_{k=1}^{K} \tau^{2}_{k}\,.
\end{equation}

\begin{theorem}
\label{thm:ge_bgl}
If $n \geq 3$, the three variable Gibbs sampler for the Bayesian group lasso is geometrically ergodic.
\end{theorem}
\begin{proof}
	See Appendix \ref{sec:bgl_proof}.
\end{proof}

\begin{remark}\label{rem:bgl_drift_analyse}
As in the Bayesian fused lasso Gibbs sampler, the drift rate, 
\[\phi_{BGL} =  \max \left\{ \dfrac{p}{n + p + 2\alpha - 2}, \dfrac{1}{2} \right\}, \] is no better than 1/2  and approaches 1 as $p$ increases.
\end{remark}

\begin{remark}\label{rem:bgl_starting_value}
Minimizing $V_{BGL}$ yields default starting values for the Markov chain as $\beta_0$ being the frequentist group lasso estimate and $\tau^2_{0,k} = 2\sqrt{\beta^T_{0,G_k} \beta_{0, G_k}}/\lambda$. See Appendix \ref{sub:bgl_starting_values} for details.
\end{remark}


\begin{remark}\label{rem:bayes_lasso}
Since for $K = p$, the Bayesian group lasso is the Bayesian lasso, our result of geometric ergodicity holds for the Bayesian lasso as well. Geometric ergodicity of the Bayesian lasso was demonstrated by \cite{khare:hobe:2013} under exactly the same conditions. Our result on the starting values in Remark \ref{rem:bgl_starting_value} also holds for the Bayesian lasso Gibbs sampler. 
\end{remark}



\section{Bayesian Sparse Group Lasso} 
\label{sec:bsgl}

The group lasso induces sparsity across groups but does not induce sparsity within a group. \citet{sim:fried:has:2013} added an $L_1$ penalty on the individual coefficients to the group lasso to arrive at the sparse group lasso. As before, for a fixed $K$, partition $\beta$ in $K$ groups each of size $m_1, m_2, \dots, m_K$, the groups being denoted by $\beta_{G_1}, \beta_{G_2}, \dots, \beta_{G_K}$. For tuning parameters $\lambda_1 >0$ and $\lambda_2 >0$, the sparse group lasso estimate is
\begin{equation}
	\label{eq:freq_sparse}
	\hat{\beta}_{\text{sgroup}}  = \arg \max_{\beta} \Big \| y - \sum_{k=1}^{K} X_{G_k}\beta_{G_k} \Big\|^2 + \lambda_1 \|\beta\|_1 + \lambda_2 \sum_{k=1}^{K} \|\beta_{G_k}\|_2 \,,
\end{equation}
where $\| \cdot\|_1$ is the $L_1$ norm. The Bayesian sparse group lasso was introduced by \cite{xu:gho:2015}. Before presenting the model, we give some definitions. Let $\gamma^2_{1,1}, \gamma^2_{1,2}, \dots, \gamma^2_{1,m_1}, \dots, \gamma^2_{K,m_K}$ and $\tau^2_1, \dots, \tau^2_p$  be variables defined on the positive reals. For each group $k$ define,
\[V_k = \text{Diag} \left\{\left(\dfrac{1}{\tau^2_k} + \dfrac{1}{\gamma^2_{k,j}} \right)^{-1} \; : \; j = 1, \dots, m_k\right\}\,. \]
The notation $\gamma^2_{k,j}$ is purely for convenience and can easily be replaced with $\gamma_{i}^2$ for $i = 1, \dots, p$. Let $\tau^2 = (\tau^2_1, \tau^2_2, \dots, \tau^2_K)$ and let $\gamma^2 = (\gamma_{1,1}^2, \dots, \gamma_{1, m_1}^2, \dots, \gamma_{K,1}^2, \dots,  \gamma^2_{K, m_K})$. The Bayesian sparse group lasso model formulated by \cite{xu:gho:2015} is
\begin{align*}
Y \mid \beta, \sigma^2 \;\; & \sim \;\;  N_n(X\beta, \sigma^2 I_n) \\
\beta_{G_k} \mid \sigma^2, \tau^2, \gamma^2 \;\; & \overset{\text{iid}}{\sim} \;\;  N_{m_k}(0, \sigma^2 V_k) \quad \text{for } k = 1, \dots, K  \numberthis \label{eq:bsgl_model}\\
\pi(\gamma_{k,1}, \dots, \gamma_{k,m_k}, \tau^2_{k}) \;\; & =  \;\; \pi_k \quad \text{ independently for }k = 1, \dots, K\\
\sigma^2  \;\; & \sim  \;\; \text{Inverse-Gamma}(\alpha, \xi)\,,
\end{align*}
where $\alpha, \xi \geq 0$ are fixed and the independent prior on each $(\gamma_{k,1}, \dots, \gamma_{k,m_k}, \tau^2_{k}) $ is
\begin{equation}
\label{eq:bsgl_prior}
\pi_k \propto \prod_{j=1}^{m_k}\left[ \left(\gamma^2_{k,j} \right)^{-\frac{1}{2}} \left(\dfrac{1}{\gamma_{k,j}^2} + \dfrac{1}{\tau_k^2}   \right)^{-\frac{1}{2}}  \right] \left(\tau^2_k \right)^{-\frac{1}{2}} \exp \left\{-\dfrac{\lambda_2^2}{2}\ds \sum_{j=1}^{m_k} \gamma^2_{k,j} - \dfrac{\lambda^2_1}{2} \tau^2_k \right\}\,.	
\end{equation}
Here  $\lambda_1, \lambda_2 > 0$ are fixed. \citet{xu:gho:2015} show that the prior in \eqref{eq:bsgl_prior} is proper with the normalizing constant being a function of $\lambda_1$ and $\lambda_2$. 

\subsection{Gibbs Sampler for Bayesian Sparse Group Lasso} 
\label{sub:gibbs_sampler_for_bsgl}
Define $V_{\tau, \gamma}$ to be the diagonal matrix with diagonals  being that of $V_1, \dots, V_K$ in that sequence. In addition, let $\beta_{k,j}$, refer to the $j$th coefficient in the $k$th group. The Bayesian sparse group lasso model in \eqref{eq:bsgl_model} leads to the following full conditionals for $\beta, \tau^2, \gamma^2$ and $\sigma^2$:
\begin{align*}
	\beta \mid \sigma^2, \tau^2, \gamma^2, y & \; \; \sim  \; \; N_p \left( (X^TX + V_{\tau, \gamma}^{-1})^{-1} X^T y , \sigma^2 (X^TX + V^{-1}_{\tau, \gamma} )^{-1}   \right)\\
	\dfrac{1}{\tau^2_{k}} \mid \beta, \sigma^2, y & \; \; \overset{\text{ind}}{\sim} \; \;  \text{Inverse-Gaussian} \left( \sqrt{\dfrac{\lambda_1^2 \sigma^2}{\beta^T_{G_k} \beta_{G_k} }}, \lambda_1^2 \right), \text{ for all $k$}\numberthis \label{eq:full_cond_bsgl}\\
	\dfrac{1}{\gamma^2_{k,j}} \mid \beta,\sigma^2, y & \; \; \overset{\text{ind}}{\sim}  \; \;  \text{Inverse-Gaussian} \left( \sqrt{\dfrac{\lambda_2^2 \sigma^2}{\beta^2_{k,j} }}, \lambda_2^2 \right), \text{ for all $k,j$}\\
	\sigma^2 \mid \beta, \tau^2, \gamma^2, y &  \; \; \sim  \; \; \text{Inverse-Gamma} \left(\dfrac{n + p + 2\alpha}{2}, \dfrac{(y - X\beta)^T(y - X\beta) + \beta^T V^{-1}_{\tau, \gamma}\beta + 2\xi }{2}   \right)\,. 
\end{align*}
Notice that the full conditionals for $\tau^2$ and $\gamma^2$ are independent and thus can be updated in one block leading to a three variable Gibbs sampler. If $(\beta_{(t)}, \tau^2_{(t)}, \gamma^2_{(t)}, \sigma^2_{(t)})$ is the current state of the Gibbs sampler, the $(t+1)$th state is obtained as follows.

\noindent\rule{15.1cm}{0.5pt}
\begin{enumerate}
	\item Draw $\sigma^{2}_{(n+1)}$ from  $f(\sigma^2 \mid \beta_{(n)}, \tau^2_{(n)}, \gamma^2_{(n)}, y)$.
	\item Draw $\left(1/\tau^2_{(n+1)}, 1/\gamma^2_{(n+1)} \right)$ from $f(1/\tau^2 \mid \beta_{(n)}, \sigma^2_{(n+1)}, y) \;f(1/\gamma^2 \mid \beta_{(n)}, \sigma^2_{(n+1)}, y)$.
	\item Draw $\beta_{(n+1)}$ from $f(\beta \mid \tau^2_{(n+1)}, \gamma^2_{(n+1)}, \sigma^2_{(n+1)}, y)$.
\end{enumerate}
\noindent\rule{15.1cm}{0.5pt}\\
The MTD for the three variable Gibbs sampler is
\begin{align*}
	& k_{BSGL}(\beta, \tau^2, \gamma^2, \sigma^2 \mid \beta_0, \tau^2_0, \gamma^2_0, \sigma^2_0)\\
	& = f(\beta \mid \tau^2, \gamma^2, \sigma^2, y) \;f(\tau^2, \gamma^2 \mid \beta_0, \sigma^2, y) \; f(\sigma^2 \mid \beta_0, \tau^2_0, \gamma_0^2, y)\,. \numberthis \label{eq:bsgl_mtd}
\end{align*}
As in the Bayesian group lasso Gibbs sampler, the MTD is strictly positive and thus aperiodic, irreducible almost everywhere, and the chain is Harris recurrent. We will prove geometric ergodicity by establishing a drift and an associated minorization condition.

Define the drift function $V_{BSGL}: \real^p \times \real^K_+ \times \real^p_+ \times \real_+ \to [0, \infty)$ as,
\begin{equation}
\label{eq:drift_bsgl}
	V_{BSGL}(\beta, \tau^2, \gamma^2, \sigma^2)  = (y - X\beta)^T (y - X\beta) + \beta^T V_{\tau, \gamma}^{-1} \beta + \dfrac{\lambda^2_1}{4} \ds \sum_{k=1}^{K} \tau^2_{k} + \dfrac{\lambda_2^2}{4} \ds \sum_{k=1}^{K} \sum_{j=1}^{m_k} \gamma^2_{k,j}\,.
\end{equation}

\begin{theorem}
\label{thm:ge_bsgl}
If $n \geq 3$, the three variable Gibbs sampler for the Bayesian sparse group lasso is geometrically ergodic.
\end{theorem}
\begin{proof}
	See Appendix \ref{sec:bsgl_proof}.
\end{proof}

\begin{remark}\label{rem:bsgl_drift_analyze}
Define $M = \max_{k} m_k$. In Appendix \ref{sub:bsgl_drift_condition} the drift rate is determined to be
\[\phi_{BSGL} = \max \left\{\dfrac{p}{n + p + 2\alpha - 2}, \dfrac{\left(1 + \dfrac{\lambda_2^2}{\lambda_1^2} \right)}{2 \left(1 + \dfrac{\lambda_1^2}{\lambda_2^2} + \dfrac{\lambda_2^2}{\lambda_1^2}  \right)}, \dfrac{\left(1 + \dfrac{\lambda_1^2}{\lambda_2^2} \right)}{2M \left(1 + \dfrac{\lambda_1^2}{\lambda_2^2} + \dfrac{\lambda_2^2}{\lambda_1^2}  \right)}    \right\}\,. \]
Unlike the drift rate in the previous two models, the drift rate here can be lower than 1/2. However, it is likely that $p$ is large enough so that $\phi_{BSGL}$ is determined by the first term $p/(n + p + 2\alpha - 2)$. In this case again, the drift rate will tend to 1 as $p$ increases and thus convergence may be slower for large $p$ problems.
\end{remark}

\begin{remark}\label{rem:starting_values_bsgl}
A reasonable starting value for this Markov chain is $\beta_0$ being the sparse group lasso estimate, $\tau^2_{0,k} = 2\sqrt{\beta^T_{0,G_k} \beta_{0,G_k} }/ \lambda_1$ and $\gamma^2_{0,k} = 2 |\beta_{0,k,j} |/ \lambda_2$. See Appendix \ref{sub:bsgl_starting_values}.
\end{remark}



\section{Discussion} 
\label{sec:discussion}


As discussed in Section \ref{sec:introduction}, reliable estimation from MCMC output rests heavily on the rate of convergence of the Markov chain. Our geometric ergodicity results immediately implies the existence of a Markov chain CLT and strong consistency of some estimators of the asymptotic covariance matrix in this CLT. As a consequence, practitioners can use tools such that effective sample size to understand the quality of the Monte Carlo estimates.

Our results of geometric ergodicity hold under reasonable conditions. We require no conditions on $p$, and only need $n$ to be larger than 3.  However, our results suggest that it may be possible for the Gibbs samplers to converge at a slower rate if $p \gg n$. This agrees with the results in \cite{raj:spark:2015}. Users might then be inclined to first use Bayesian variable selection alternatives to these models. For example, \cite{xu:gho:2015} introduced the Bayesian variable selection alternatives to the group and the sparse group lasso by using spike-and-slab type priors. A natural direction for future research would be to investigate the convergence rate for the Gibbs samplers in these Bayesian variable selection models. 
\subsection*{Acknowledgements} 
\label{sub:acknowledgements}

The author is grateful to Galin Jones for helpful conversations and suggestions and Sakshi Arya for proof reading. The author was supported by the Alumni Fellowship, School of Statistics, University of Minnesota. 





\appendix

\section{Preliminaries} 
\label{sec:preliminaries}

In general, $\E_{(k)}$ represents expectation with respect to the MTD being studied in the section. Expectations with respect to a full conditional is  denoted by $\E_{\cdot}$. The index 0 on variables denotes starting values for the Markov chain.

Below are some properties of known distributions that will be used often.

\begin{itemize}
	\item If $1/X \sim $ Inverse-Gaussian$(a, b)$, then $\E[X] = 1/a + 1/b$.
	\item If $X \sim N_p(\mu, \Sigma)$, then $\E[XX^T] = \Sigma + \mu \mu^T$.
	\item If $X \sim $ Inverse-Gamma$(a,b)$, then $\E[X] = b/(a-1)$.
	\item If $X \sim $ Inverse-Gamma$(a,b)$, then $\E[1/X] = a/b$.
\end{itemize}
\label{sub:properties_of_distributions}


\subsection{Useful Lemmas} 
\label{sub:useful_lemmas}

We present some results that will used in the proofs of geometric ergodicity for all three samplers. Most of the results are generalizations of the results in \cite{khare:hobe:2013} and the proofs are presented here for completeness.
\begin{lemma}
\label{lem:beta_trace_bound}
Let $y, X,$ and $\beta$ be the observed $n \times 1$ response, the $n \times p$ matrix of covariates and the $p \times 1$ vector of regression coefficients. Let $\Sigma$ be the $p \times p$ positive definite matrix such that
\[\beta \sim N_p\left( (X^TX + \Sigma^{-1})^{-1} X^Ty, \sigma^2 (X^TX + \Sigma^{-1})^{-1} \right)\,, \] 
for $\sigma^2 >0$. Then,
\[\E \left[(y - X\beta)^T (y - X\beta) + \beta^T\Sigma^{-1} \beta \right] \leq y^Ty + p\sigma^2.\, \]
\end{lemma}
\begin{proof}
Consider,
\begin{align*}
&\E \left[(y - X\beta)^T (y - X\beta) + \beta^T\Sigma^{-1} \beta \right] \\
& = y^Ty - 2y^TX \, \E\left[\beta \right] +  \E \left[ \beta^T(X^TX + \Sigma^{-1})\beta \right]\\
& = y^Ty - 2y^TX (X^TX + \Sigma^{-1})^{-1}X^Ty + \E \left[ tr(\beta^T(X^TX + \Sigma^{-1})\beta) \right]\\
& = y^Ty - 2y^TX (X^TX + \Sigma^{-1})^{-1}X^Ty + \, tr\left(\sigma^2 (X^TX + \Sigma^{-1}) (X^TX + \Sigma^{-1})^{-1} \right) \\
& \quad + \, tr\left((X^TX + \Sigma^{-1})(X^TX + \Sigma^{-1})^{-1}X^Tyy^TX(X^TX + \Sigma^{-1})^{-1} \right)\\
& = y^Ty - 2y^TX (X^TX + \Sigma^{-1})^{-1}X^Ty + p\sigma^2 + \,tr\left(y^TX(X^TX + \Sigma^{-1})^{-1}X^Ty \right)\\
& \leq y^Ty + p\sigma^2\,.
\end{align*}
\end{proof}

\begin{lemma}
\label{lem:tau_fraction_bound}
For $\alpha = (\alpha_1, \dots, \alpha_p) \in \real^p$ and $\delta = (\delta_1, \dots, \delta_p)$ such that $\delta_i \ne 0$,
\[\dfrac{\sum_{i=1}^{p} \alpha^2_i}{\sum_{i=1}^{p} \alpha^2_i/\delta^2_i} \leq \ds \sum_{i=1}^{p} \delta_i^2. \]
\end{lemma}
\begin{proof}
Using the fact that the square of a number is non-negative, 
\begin{align*}
\dfrac{\sum_{i=1}^{p} \alpha^2_i}{\sum_{i=1}^{p} \alpha^2_i/\delta^2_i} &= \dfrac{\ds \sum_{i=1}^{p} \dfrac{\alpha^2_i}{\delta^2_i} \delta^2_i}{\ds \sum_{i=1}^{p} \alpha^2_i/\delta^2_i} \leq \dfrac{\ds \sum_{i=1}^{p} \dfrac{\alpha^2_i}{\delta^2_i} \left(\ds \sum_{i=1}^{p} \delta^2_i \right)}{\ds \sum_{i=1}^{p} \alpha^2_i/\delta^2_i} = \ds \sum_{i=1}^{p} \delta^2_i\,.
\end{align*}
\end{proof}

\begin{lemma}
\label{lem:inv_gau_dist}
For $\lambda^2, a^2, \sigma^2 > 0$, if $X$ has a probability density function $f(x)$ such that
\[f(x) \propto x^{-1/2} \exp \left\{ -\dfrac{\lambda^2 x}{2} - \dfrac{a^2}{2 \sigma^2 x} \right\}\,, \] then $1/X \sim $ Inverse-Gaussian distribution with mean parameter $\sqrt{\lambda^2 \sigma^2/a^2}$ and scale parameter $\lambda^2$.
\end{lemma}
\begin{proof}
For the change of variable $z = 1/x$,
\begin{align*}
f(z) & \propto z^{-2} z^{\frac{1}{2}} \exp \left\{ - \dfrac{\lambda^2}{2 z} - \dfrac{a^2 z}{2 \sigma^2}   \right\}
= z^{-\frac{3}{2}}\exp \left\{ - \dfrac{a^2 \left(\frac{\lambda^2\sigma^2}{a^2} + z^2 \right)}{2 \sigma^2 z}  \right\}\\
& = \exp \left\{- \sqrt{\dfrac{\lambda^2 a^2}{\sigma^2} }  \right\} z^{-\frac{3}{2}}\exp \left\{ - \dfrac{a^2 \left(\frac{\lambda^2\sigma^2}{a^2} - 2\sqrt{\frac{\lambda^2\sigma^2}{a^2}}z + z^2 \right)}{2 \sigma^2 z}  \right\} \\
& \propto z^{-\frac{3}{2}} \exp \left\{-\dfrac{\lambda^2  \left(\frac{\lambda^2\sigma^2}{a^2} - 2\sqrt{\frac{\lambda^2\sigma^2}{a^2}}z + z^2 \right) }{2 \frac{\lambda^2 \sigma^2}{a^2 }z }  \right\}\,.
\end{align*}
Thus, $Z \sim $ Inverse-Gaussian with mean parameter $\sqrt{\lambda^2 \sigma^2/a^2}$ and scale parameter $\lambda^2$.
\end{proof}

\begin{lemma}
\label{lem:inv_gauss_minor}
If $1/X \sim $ Inverse-Gaussian with mean parameter $\sqrt{\lambda^2 \sigma^2/a^2}$ and scale parameter $\lambda^2$ and $a^2 \leq d^2$ for some $d^2 > 0$, then
\[f(x) \geq \exp \left\{-\sqrt{\dfrac{\lambda^2 d^2}{\sigma^2} } \right\} q(x)\,, \]
where $f(x)$ is the pdf of $X$ and $q(x)$ is the pdf of the reciprocal of the Inverse-Gaussian distribution with mean parameter $\sqrt{\lambda^2 \sigma^2/d^2}$ and scale parameter $\lambda^2$.
\end{lemma}
\begin{proof}
By Lemma \ref{lem:inv_gau_dist}, we have
\begin{align*}
 f(x) & =  \sqrt{\dfrac{\lambda^2}{2\pi}} \left(x\right)^{-\frac{1}{2}} \exp \left\{-\dfrac{\lambda^2  \left(\dfrac{\lambda^2\sigma^2}{a^2} - 2\sqrt{\dfrac{\lambda^2\sigma^2}{a^2}}\dfrac{1}{x} + \dfrac{1}{x^2} \right) }{2 \dfrac{\lambda^2 \sigma^2}{a^2 }\dfrac{1}{x} }  \right\}\\
 & =  \exp \left\{ \sqrt{\dfrac{\lambda^2 a^2}{\sigma^2} }   \right\} \sqrt{\dfrac{\lambda^2}{2\pi}} \left(x\right)^{-\frac{1}{2}} \exp \left\{-\dfrac{\lambda^2  \left(\dfrac{\lambda^2\sigma^2}{a^2} + \dfrac{1}{x^2} \right) }{2 \dfrac{\lambda^2 \sigma^2}{a^2 }\dfrac{1}{x} }  \right\}\\
 & \geq  \sqrt{\dfrac{\lambda^2}{2\pi}} \left(x\right)^{-\frac{1}{2}} \exp \left\{-\dfrac{\lambda^2  \left(\dfrac{\lambda^2\sigma^2}{a^2} + \dfrac{1}{x^2} \right) }{2 \dfrac{\lambda^2 \sigma^2}{a^2 }\dfrac{1}{x} }  \right\}
 = \exp \left\{- \sqrt{\dfrac{\lambda^2 d^2}{\sigma^2} }   \right\}  q(x)\,.
\end{align*}
\end{proof}

\begin{lemma}
\label{lem:yty_bound_minor}
Let $y, X,$ and $\beta$ be the observed $n \times 1$ response, the $n \times p$ matrix of covariates and the $p \times 1$ vector of regression coefficients respectively. Let $\Sigma$ be a $p \times p$ positive definite matrix. Then,
\[(y - X\beta)^T(y - X\beta) + \beta^T \Sigma^{-1} \beta \geq y^Ty - y^T X \left( X^TX + \Sigma^{-1} \right)^{-1} X^Ty\,. \]
\end{lemma}
\begin{proof}
The proof mainly requires completing the square in the following way,
\begin{align*}
& (y - X\beta)^T(y - X\beta) + \beta^T \Sigma^{-1}\beta \\
& = y^Ty - 2y^TX(X^TX + \Sigma^{-1})(X^TX + \Sigma^{-1})^{-1}\beta + \beta^T(X^TX + \Sigma^{-1}) \beta \\
& \quad +  y^TX (X^TX + \Sigma^{-1})^{-1}(X^TX + \Sigma^{-1})(X^TX + \Sigma^{-1})^{-1}X^Ty\\
& \quad  -  y^TX (X^TX + \Sigma^{-1})^{-1}(X^TX + \Sigma^{-1})(X^TX + \Sigma^{-1})^{-1}X^Ty\\
& = y^Ty -  y^TX(X^TX + \Sigma^{-1})^{-1}X^Ty\\
& \quad + (\beta - (X^TX + \Sigma^{-1})^{-1} X^Ty)^T (X^TX + \Sigma^{-1}) (\beta - (X^TX + \Sigma^{-1})^{-1} X^Ty)\\
& \geq y^Ty - y^TX(X^TX + \Sigma^{-1})^{-1} X^Ty\, .
\end{align*}
\end{proof}


\section{Bayesian Fused Lasso Prior} 
\label{sec:bfl_prior}

\subsection{Propriety of the Prior} 
\label{sub:propriety_of_prior}
First note that $\det\left(\Sigma_{\tau, w}\right) = \left(\det \left( \Sigma_{\tau, w}^{-1} \right) \right)^{-1}$. We decompose $\Sigma^{-1}_{\tau, w}$ into 
\begin{equation}
\label{eq:sigma_split}
	 \Sigma^{-1}_{\tau, w} = L_1 + L_2\,,
\end{equation}
where 
\[L_1 = \text{diag}\left(\dfrac{1}{2\tau^2_1}, \dfrac{1}{2\tau^2_2}, \dots, \dfrac{1}{2\tau^2_p} \right) \; \text{and}\,, \]
\begin{equation*}
L_2 = \left[\begin{array}{ccccc} \frac{1}{2\tau^2_1} + \frac{1}{w^2_1} & -\frac{1}{w^2_1} & 0 & \dots & 0 \\ -\frac{1}{w^2_1} & \frac{1}{2\tau^2_2} + \frac{1}{w^2_1} + \frac{1}{w^2_2} & -\frac{1}{w^2_2} & \dots  & 0 \\ 0 & -\frac{1}{w_2^2} & \frac{1}{2\tau^2_3} + \frac{1}{w^2_2} + \frac{1}{w^2_3} & \dots &  0\\ \dots & \dots & \dots & \ddots & \dots \\ 0 & 0 & \dots & \frac{1}{2\tau^2_{p-1}} + \frac{1}{w^2_{p-2}} + \frac{1}{w^2_{p-1}}  & -\frac{1}{w^2_{p-1}}  \\  0 & 0 & \dots & -\frac{1}{w^2_{p-1}} & \frac{1}{2\tau^2_p} + \frac{1}{w^2_{p-1}} \end{array}   \right]\,.
\end{equation*}
The diagonal matrix $L_1$ is clearly positive definite. The tridiagonal matrix $L_2$ is also positive definite since $L_2$ is real symmetric, has positive diagonals, and is strictly diagonally dominant \cite[Theorem 1.2]{andj:dafo:2011}. Here the condition of strict diagonal dominance is satisfied since 
\[ \dfrac{1}{2 \tau^2_{i} } + \dfrac{1}{w^2_{i-1}} + \dfrac{1}{w^2_{i}} > \dfrac{1}{w^2_{i-1}} + \dfrac{1}{w^2_{i}}\,. \] 

Thus, 
\begin{align*}
	\det(\Sigma^{-1}_{\tau, w}) & = \det(L_1 + L_2) \geq \det(L_1) + \det(L_2) \geq \det(L_1) = \prod_{i=1}^{p} \left(\dfrac{1}{2 \tau^2_i}  \right)\\
	\Rightarrow \;\left(\det\left( \Sigma^{-1}_{\tau, w} \right)\right)^{-1/2} & \leq \prod_{i=1}^{p} \left( 2\tau^2_{i} \right)^{1/2}\,.
\end{align*}

Thus, the joint prior on $(\tau^2, w^2)$ satisfies,
\begin{align*}
	\pi(\tau^2, w^2) & \propto \det\left(\Sigma_{\tau, w} \right)^{1/2} \left( \prod_{i=1}^{p} \left(\tau^2_i \right)^{-1/2} e^{-\lambda_1 \tau^2_i/2} \right) \left(\prod_{i=1}^{p-1} \left(w^2_i \right)^{-1/2}e^{-\lambda_2 w^2_i/2} \right)\\
	& \leq \prod_{i=1}^{p} \left( 2\tau^2_{i} \right)^{1/2} \left( \prod_{i=1}^{p} \left(\tau^2_i \right)^{-1/2} e^{-\lambda_1 \tau^2_i/2} \right) \left(\prod_{i=1}^{p-1} \left(w^2_i \right)^{-1/2}e^{-\lambda_2 w^2_i/2} \right)\\
	& = 2^{p/2}  \left( \prod_{i=1}^{p} e^{-\lambda_1 \tau^2_i/2} \right) \left(\prod_{i=1}^{p-1} \left(w^2_i \right)^{-1/2}e^{-\lambda_2 w^2_i/2} \right)\,.
\end{align*}
This is the product of $p$ exponentials densities and $p-1$ Gamma densities. Thus, the prior is proper.


\subsection{Validity of the Prior} 
\label{sub:validity_of_the_prior}

In this section we demonstrate that our choice of prior in the Bayesian fused lasso leads to the Laplace prior in \eqref{eq:laplace_fused}. First we expand $\beta^T \Sigma^{-1}_{\tau, w} \beta$ in the following way:
\begin{align*}
\beta ^T \Sigma^{-1}_{\tau, w} \beta &= \left[\begin{array}{c}
\beta_1 \left(\dfrac{1}{\tau^2_1} + \dfrac{1}{w_1^2} \right) - \dfrac{\beta_2}{w_1^2} \\
-\dfrac{\beta_1}{w_1^2} + \beta_2 \left(\dfrac{1}{\tau^2_2} + \dfrac{1}{w_1^2} + \dfrac{1}{w_2^2} \right) - \dfrac{\beta_3}{w_2^2}\\
-\dfrac{\beta_2}{w_2^2} + \beta_3\left(\dfrac{1}{\tau^2_3} + \dfrac{1}{w^2_2} + \dfrac{1}{w_3^2}  \right) - \dfrac{\beta_4}{w_3^2}\\
 \vdots\\
 -\dfrac{\beta_{p-1}}{w^2_{p-1}} +  \beta_p \left(\dfrac{1}{\tau^2_p}  + \dfrac{1}{w^2_{p-1}} \right)
\end{array} \right]^T  \left[\begin{array}{c}
	\beta_1 \\ \beta_2 \\ \beta_3 \\ \vdots \\ \beta_p
\end{array} \right]\\
& = \beta_1^2 \left(\dfrac{1}{\tau^2_1} + \dfrac{1}{w_1^2} \right) - \dfrac{\beta_1\beta_2}{w_1^2} -\dfrac{\beta_1 \beta_2}{w_1^2} + \beta_2^2 \left(\dfrac{1}{\tau^2_2} + \dfrac{1}{w_1^2} + \dfrac{1}{w_2^2} \right) - \dfrac{\beta_2\beta_3}{w_2^2}\\
& \quad -\dfrac{\beta_2\beta_3}{w_2^2} + \beta_3^2\left(\dfrac{1}{\tau^2_3} + \dfrac{1}{w^2_2} + \dfrac{1}{w_3^2}  \right) - \dfrac{\beta_3\beta_4}{w_3^2} + \dots  -\dfrac{\beta_{p-1} \beta_p }{w^2_{p-1}} +  \beta^2_p \left(\dfrac{1}{\tau^2_p}  + \dfrac{1}{w^2_{p-1}} \right) \\
& = \ds \sum_{i=1}^{p} \dfrac{\beta_i^2}{\tau^2_i} + \dfrac{\beta_1^2 + \beta_2^2 - 2 \beta_1 \beta_2}{w_1^2} + \dfrac{\beta_2^2 + \beta_3^2 - 2 \beta_2 \beta_3}{w_2^2} + \dots + \dfrac{\beta_p^2 + \beta_{p-1}^2 - 2 \beta_p \beta_{p-1}}{w^2_{p-1}}\\
& = \ds \sum_{i=1}^{p} \dfrac{\beta_i^2}{\tau^2_i}  + \ds \sum_{i=1}^{p-1} \dfrac{(\beta_{i+1} - \beta_{i})^2}{w_i^2} \numberthis \label{eq:fused_beta_expansion}\,.
\end{align*}
Using \eqref{eq:fused_beta_expansion},
\begin{align*}
	& \pi(\beta \mid \sigma^2) \\ 
	&\propto \int_{\real_+^p} \int_{\real_+^{p-1}} (2\pi \sigma^2)^{-\frac{p}{2} } \det\left( \Sigma^{-1}_{\tau, w} \right)^{1/2}\exp \left\{-\dfrac{\beta^T \Sigma^{-1}_{\tau, w} \beta}{2 \sigma^2}   \right\} \\
	& \quad  \times  \; \;\det\left(\Sigma_{\tau, w} \right)^{1/2} \left( \prod_{i=1}^{p} \left(\tau^2_i \right)^{-1/2} e^{-\lambda_1 \tau^2_i/2} \right) \left(\prod_{i=1}^{p-1} \left(w^2_i \right)^{-1/2}e^{-\lambda_2 w^2_i/2} \right) d w^2 d \tau^2\\
	& \propto \int \prod_{i=1}^{p} \left(\tau^2_i\right)^{-1/2} \exp \left\{-\dfrac{\lambda_1\tau^2_i}{2} - \dfrac{\beta_i^2}{2 \sigma^2 \tau^2_i}   \right\} d\tau^2 \int \prod_{i=1}^{p-1} \left(w^2_i\right)^{-1/2} \exp \left\{-\dfrac{\lambda_2w^2_i}{2} - \dfrac{ \left(\beta_{i+1} - \beta_{i} \right)^2 }{2 \sigma^2 w^2_i}   \right\} dw^2\\
	& = \exp \left\{- \dfrac{\lambda_1}{\sigma}\ds \sum_{i=1}^{p} |\beta_i|  - \dfrac{\lambda_2}{\sigma}\ds \sum_{i=1}^{p-1} |\beta_{i+1} - \beta_i| \right\} \int \prod_{i=1}^{p} \left(\tau^2_i\right)^{-1/2} \exp \left\{-\dfrac{\lambda_1\tau^2_i}{2} - \dfrac{\beta_i^2}{2 \sigma^2 \tau^2_i} + \dfrac{\lambda_1}{\sigma} |\beta_i|    \right\} d\tau^2 \\
	& \quad \times \int \prod_{i=1}^{p-1} \left(w^2_i\right)^{-1/2} \exp \left\{-\dfrac{\lambda_2w^2_i}{2} - \dfrac{ \left(\beta_{i+1} - \beta_{i} \right)^2 }{2 \sigma^2 w^2_i} + \dfrac{\lambda_2}{\sigma} |\beta_{i+1} - \beta_i|    \right\} dw^2 \\
	& \propto \exp \left\{- \dfrac{\lambda_1}{\sigma}\ds \sum_{i=1}^{p} |\beta_i|  - \dfrac{\lambda_2}{\sigma}\ds \sum_{i=1}^{p-1} |\beta_{i+1} - \beta_i| \right\}\,,
\end{align*}
where the last equality is due to the integrands being the densities of the reciprocal of Inverse-Gaussian distributions; see Lemma \ref{lem:inv_gau_dist}. 

\section{Proof of Geometric Ergodicity in Bayesian Fused Lasso} 
\label{sec:bfl_proof}

We will establish geometric ergodicity of the three variable Gibbs sampler for the Bayesian fused lasso by establishing a drift condition and an associated minorization condition.

\subsection{Drift Condition} 
\label{sub:bfl_drift_condition}

Consider the drift function 
\begin{equation}
\label{eq:drift_function_2}
	V_{BFL}(\beta, \tau^2, w^2, \sigma^2) = (y - X\beta)^T(y-X\beta) + \beta^T \Sigma_{\tau, w}^{-1} \beta + \dfrac{\lambda_1^2}{4}\ds \sum_{i=1}^{p} \tau^2_i + \dfrac{\lambda_2^2}{4}\ds \sum_{i=1}^{p-1}w^2_{i}\,.
\end{equation}
Then $V_{BFL}:\real^p \times \real^p_+ \times \real^{p-1}_+ \times \real_+ \to [0, \infty)$. To establish the drift condition we need to show that there exists a $0 < \phi_{BFL} < 1$ and $L_{BFL} > 0$ such that,
\[\E_{(k)}\left[ V_{BFL}(\beta, \tau^2, w^2, \sigma^2) \, | \, \beta_0, \tau^2_0, w^2_0, \sigma^2_0   \right] \leq \phi_{BFL} V_{BFL}(\beta_0, \tau^2_0, w_0^2, \sigma^2_0) + L_{BFL}\,, \]
for every $(\beta_0, \tau^2_0, w_0^2, \sigma^2_0) \in \real^p \times \real_+^p \times \real_+^{p-1} \times \real_+$. The left hand side is the expectation with respect to the MTD, that is,
\begin{align*}
	&\E_{(k)}\left[ V_{BFL}(\beta, \tau^2, w^2,  \sigma^2) \, \mid \, \beta_0, \tau^2_0, w^2_0, \sigma^2_0   \right] \\
	 &= \int  V_{BFL}(\beta, \tau^2, w^2, \sigma^2) f(\sigma^2 \mid \beta_0, \tau^2_0, w^2_0, y)  f(\tau^2, w^2 \mid \beta_0, \sigma^2,y) f(\beta \mid \tau^2, w^2, \sigma^2,y) d\beta \, d\tau^2 \, dw^2 \, d\sigma^2\\
	 &= \int f(\sigma^2 | \beta_0, \tau^2_0, w_0^2, y) \int f(\tau^2 , w^2 |\beta_0,  \sigma^2,y)
 \int V_{BFL}(\beta, \tau^2, w^2, \sigma^2) f(\beta | \tau^2, w^2,  \sigma^2, y) d \beta \, d\tau^2 \, dw^2 \, d\sigma^2\\
	 & =  \E_{\sigma^2} \left[  \E_{\tau^2, w^2} \left[  \E_{\beta} \left[V_{BFL}(\beta, \tau^2, w^2, \sigma^2)  \, \mid \, \tau^2,w^2, \sigma^2, y \right]  \, \mid \, \beta_0, \sigma^2, y \right] \, \mid \, \beta_0, \tau^2_0, w^2_0,  y  \right]\,.
\end{align*}
We will evaluate these sequentially, starting with the innermost expectation.  By Lemma \ref{lem:beta_trace_bound},
\begin{align*}
\E_{(k)} \left[V_{BFL}(\beta, \tau^2, w^2, \sigma^2) \mid \tau^2, w^2,  \sigma^2, y \right] 
& \leq y^Ty + \dfrac{\lambda_1^2}{4} \ds \sum_{i=1}^{p} \tau^2_i + \dfrac{\lambda_2^2}{4}\ds \sum_{i=1}^{p-1}w^2_{i} + p\sigma^2\,.
\end{align*}

Next we move on to the expectation with respect to the full conditional of $\tau^2, w^2$. Note that
\begin{align*}
& \E_{\tau^2, w^2} \left[  \E_{\beta} \left[V_{BFL}(\beta, \tau^2, w^2,  \sigma^2)  \, \mid \, \tau^2, w^2, \sigma^2, y \right]  \, \mid \, \beta_0, \sigma^2, y \right] \\
& \leq y^Ty + p \sigma^2 + \dfrac{\lambda_1^2}{4} \sum_{i=1}^{p} \left[ \sqrt{ \dfrac{\beta^2_{0,i} }{\lambda_1^2 \sigma^2} } + \dfrac{1}{\lambda_1^2}  \right] + \dfrac{\lambda^2_2}{4} \ds \sum_{i=1}^{p-1} \left[ \sqrt{\dfrac{(\beta_{0,i+1} - \beta_{0,i} )^2}{\lambda_2^2 \sigma^2} } + \dfrac{1}{\lambda^2_2} \right]\,, \numberthis \label{eq:bfl_till_tau_1}
\end{align*}
using the properties of the Inverse-Gaussian distribution mentioned in Appendix \ref{sub:properties_of_distributions}. Since for $a,b > 0$, $2ab \leq a^2 + b^2$, 
\begin{align*}
& \E_{\tau^2, w^2} \left[  \E_{\beta} \left[V_{BFL}(\beta, \tau^2, w^2,  \sigma^2)  \, \mid \, \tau^2, w^2, \sigma^2, y \right]  \, \mid \, \beta_0, \sigma^2, y \right] \\
& \leq y^Ty + p \sigma^2 + \dfrac{\lambda_1^2}{4} \sum_{i=1}^{p} \left[ \dfrac{ \beta_{0,i}^2}{2 \sigma^2 (n + p + 2\alpha)}  + \dfrac{(n + p + 2\alpha)}{2 \lambda_1^2} + \dfrac{1}{\lambda_1^2} \right]\\
& \quad + \dfrac{\lambda^2_2}{4} \ds \sum_{i=1}^{p-1} \left[ \dfrac{(\beta_{0,i+1} - \beta_{0,i} )^2}{2 \sigma^2 (n + p + 2\alpha)} + \dfrac{(n + p + 2\alpha)}{2 \lambda^2_2} + \dfrac{1}{\lambda^2_2} \right]\\
& \leq y^Ty + \dfrac{p}{4} \left(2 + (n + p + 2\alpha)  \right) + p \sigma^2 + \dfrac{\lambda_1^2}{8 (n + p + 2\alpha)} \ds \sum_{i=1}^{p} \dfrac{\beta^2_{0,i} }{\sigma^2} + \dfrac{\lambda_2^2 }{8 (n + p + 2\alpha) }  \ds \sum_{i=1}^{p-1}  \dfrac{(\beta_{0,i+1} - \beta_{0,i})^2}{\sigma^2}\,.
\end{align*}
Finally, the last expectation,
\begin{align*}
& \E_{\sigma^2} \left[  \E_{\tau^2, w^2} \left[  \E_{\beta} \left[V_{BFL}(\beta, \tau^2, w^2,  \sigma^2)  \, \mid \, \tau^2, w^2, \sigma^2, y \right]  \, \mid \, \beta_0, \sigma^2, y \right] \, \mid \, \beta_0, \tau^2_0, w^2_0, y  \right]\\
& \leq y^Ty  + \dfrac{p}{4} \left(n + p + 2\alpha + 2\right) + p \E_{\sigma^2}[\sigma^2 \mid \beta_0, \tau^2_0, w_0^2, y]+ \dfrac{\lambda_1^2}{8 (n + p + 2\alpha)} \ds \sum_{i=1}^{p} \E_{\sigma^2} \left[\dfrac{\beta^2_{0,i} }{\sigma^2} \mid \beta_0, \tau^2_0, w^2_0, y \right]\\
& \quad + \dfrac{\lambda_2^2}{8 (n + p + 2\alpha)} \ds \sum_{i=1}^{p-1} \E_{\sigma^2} \left[\dfrac{(\beta_{0,i+1} - \beta_{0,i} )^2  }{\sigma^2} \mid \beta_0, \tau^2_0, w^2_0, y \right]\\
& \leq y^Ty + \dfrac{p}{4} \left(n + p + 2\alpha + 2 \right) + p  \dfrac{(y - X\beta_0)^T(y - X\beta_0) + \beta_0^T \Sigma^{-1}_{\tau_0, w_0} \beta_0 + 2 \xi}{n + p + 2\alpha - 2}  \\
& \quad + \dfrac{\lambda_1^2}{8 (n + p + 2\alpha)} \ds \sum_{i=1}^{p} \dfrac{ (n + p + 2\alpha)\beta^2_{0,i} }{(y - X\beta_0)^T(y - X\beta_0) + \beta_0^T \Sigma^{-1}_{\tau_0, w_0} \beta_0 + 2\xi} \\
& \quad + \dfrac{\lambda_2^2}{8 (n + p + 2\alpha)} \ds \sum_{i=1}^{p-1} \dfrac{ (n + p + 2\alpha)(\beta_{0,i+1} - \beta_{0,i} )^2 }{(y - X\beta_0)^T(y - X\beta_0) + \beta_0^T \Sigma^{-1}_{\tau_0, w_0} \beta_0 + 2\xi} \\
& \leq y^Ty + p \dfrac{(y - X\beta_0)^T(y - X\beta_0) + \beta_0^T \Sigma^{-1}_{\tau_0, w_0} \beta_0 + 2 \xi}{n + p + 2\alpha - 2} \\
& \quad + \dfrac{p}{4}(n + p + 2\alpha + 2) + \dfrac{\lambda^2_1}{8} \dfrac{\sum_{i=1}^{p} \beta^2_{0,i} }{\beta_0^T \Sigma^{-1}_{\tau_0, w_0} \beta_0} + \dfrac{\lambda_2^2}{8 }   \dfrac{ \sum_{i=1}^{p-1}(\beta_{0,i+1} - \beta_{0,i} )^2  }{\beta_0^T \Sigma^{-1}_{\tau_0, w_0} \beta_0}\,. \numberthis \label{eq:bfl_till_sigma_1}
\end{align*}
Using \eqref{eq:fused_beta_expansion},
\begin{equation}
	\label{eq:bfl_betas_ineq}
\beta_0^T \Sigma^{-1}_{\tau_0, w_0} \beta_0 \geq \ds \sum_{i=1}^{p} \dfrac{\beta_{0,i}^2}{\tau^2_{0,i}} \quad \text{ and } \quad \beta_0^T \Sigma^{-1}_{\tau_0, w_0} \beta_0 \geq \ds \sum_{i=1}^{p-1} \dfrac{(\beta_{0,i+1} - \beta_{0,i} )^2}{w^2_{0,i}}\,.
\end{equation}
Using \eqref{eq:bfl_betas_ineq} in \eqref{eq:bfl_till_sigma_1},
\begin{align*}
& \E_{\sigma^2} \left[  \E_{\tau^2, w^2} \left[  \E_{\beta} \left[V_{BFL}(\beta, \tau^2, w^2,  \sigma^2)  \, \mid \, \tau^2, w^2, \sigma^2, y \right]  \, \mid \, \beta_0, \sigma^2, y \right] \, \mid \, \beta_0, \tau^2_0, w^2_0, y  \right]\\
& \leq y^Ty + p \dfrac{(y - X\beta_{0})^T(y - X\beta_{0}) + \beta_{0}^T \Sigma^{-1}_{\tau_{0}, w_0} \beta_{0} + 2 \xi}{n + p + 2\alpha - 2}\\
& \quad  + \dfrac{p}{4}(n + p + 2\alpha + 2) + \dfrac{\lambda^2_1}{8} \dfrac{\sum_{i=1}^{p} \beta^2_{0,i} }{ \sum_{i=1}^{p}\beta^2_{0,i}/\tau^2_{0,i} } + \dfrac{\lambda_2^2}{8 }   \dfrac{ \sum_{i=1}^{p-1}(\beta_{0,i+1} - \beta_{0,i})^2  }{ \sum_{i=1}^{p-1} (\beta_{0,i+1} - \beta_{0,i})^2/w_{0,i}^2}\,.
\end{align*}
By Lemma \ref{lem:tau_fraction_bound},
\begin{align*}
& \E_{\sigma^2} \left[  \E_{\tau^2, w^2} \left[  \E_{\beta} \left[V_{BFL}(\beta, \tau^2, w^2,  \sigma^2)  \, \mid \, \tau^2, w^2, \sigma^2, y \right]  \, \mid \, \beta_0, \sigma^2, y \right] \, \mid \, \beta_0, \tau^2_0, w^2_0, y  \right]\\
& \leq y^Ty + \dfrac{p}{4}(n + p + 2\alpha + 2) + \dfrac{2p\xi}{n + p + 2\alpha - 2} \\
& \quad + \dfrac{p}{n+p+2\alpha - 2} \left((y - X\beta_{0})^T(y - X\beta_{0}) + \beta_{0}^T \Sigma^{-1}_{\tau_{0}, w} \beta_{0} \right) + \dfrac{\lambda^2_1}{8} \ds \sum_{i=1}^{p} \tau^2_{0,i} + \dfrac{\lambda_2^2}{8}\ds \sum_{i=1}^{p-1} w_{0,i}^2\\
& \leq \phi_{BFL} V(\beta_0, \tau^2_0, w^2_0, \sigma^2_0) + L_{BFL}\,,
\end{align*}
where
\begin{equation}
\label{eq:phi_bsgl}
\phi_{BFL} = \max \left\{\dfrac{p}{n + p + 2\alpha - 2}, \dfrac{1}{2}  \right\} < 1 \text{ for } n \geq 3\quad \text{ and }
\end{equation}
\begin{equation}
\label{eq:L_bsgl}
L_{BFL} = y^Ty + \dfrac{p}{2} \left(n + p + 2\alpha + 2 \right) + \dfrac{2 p \xi }{n + p + 2 \alpha  - 2}\,.
\end{equation}

\subsection{Minorization} 
\label{sub:minorization}

To establish a one-step minorization, we need to show that for all sets $C_d$ defined as 
\[C_d =  \{ (\beta, \tau^2, w^2,  \sigma^2) : V_{BFL}(\beta, \tau^2, w^2, \sigma^2) \leq d\}\, , \]
there exists an $\epsilon > 0$ and a density $q$ such that for all $(\beta_0, \tau^2_0, w^2_0, \sigma^2_0) \in C_d$
\[k_{BFL}(\beta, \tau^2, w^2, \sigma^2 \mid \beta_0, \tau^2_0, w^2_0, \sigma^2_0) \geq \epsilon\, q(\beta, \tau^2, w^2, \sigma^2)\,. \]
To establish this condition, recall that,
\begin{align*}
k_{BFL}(\beta, \tau^2, w^2, \sigma^2 \mid \beta_0, \tau^2_0, w^2_0, \sigma^2_0) & = f(\beta \mid \tau^2, w^2, \sigma^2, y) \;f(\tau^2, w^2 \mid \beta_0, \sigma^2, y) f(\sigma^2 \mid \beta_0, \tau^2_0, w_0^2, y)\,.
\end{align*}
For our drift function, for all $(\beta_0, \tau^2_0, w_0^2, \sigma_0^2) \in C_d$ the following relation holds due to \eqref{eq:fused_beta_expansion}:
\begin{align*}
(y - X\beta_0)^T(y - X\beta_0) + \beta_0^T \Sigma_{\tau_0, w_0}^{-1} \beta_0 + \dfrac{\lambda_1^2}{4}\ds \sum_{i=1}^{p} \tau^2_{0,i}  + \dfrac{\lambda_2^2}{4} \ds \sum_{i=1}^{p-1} w^2_{0,i} & \leq d\\
(y - X\beta_0)^T(y - X\beta_0) + \ds \sum_{i=1}^{p} \dfrac{\beta_{0,i}^2}{\tau^2_{0,i} } + \ds \sum_{i=1}^{p-1} \dfrac{(\beta_{0,i+1} - \beta_{0,i})^2}{w_{0,i}^2}+ \dfrac{\lambda_1^2}{4}\ds \sum_{i=1}^{p} \tau^2_{0,i}  + \dfrac{\lambda_2^2}{4} \ds \sum_{i=1}^{p-1} w^2_{0,i} & \leq d\,.
\end{align*}
Using the above and Lemma \ref{lem:tau_fraction_bound}, for each $\beta_{0,j}$,
\begin{align*}
	 & \beta_{0,j}^2 \leq \sum_{i=1}^{p} \beta_{0,i}^2 \leq \left( \ds \sum_{i=1}^{p} \tau^2_{0,i} \right) \left(\ds \sum_{i=1}^{p} \dfrac{\beta_{0,i}^2}{\tau^2_{0,i}}  \right) \leq \dfrac{4d^2}{\lambda_1^2} := d_1^2 \numberthis \label{eq:tau_beta_bound_bfl}\,,
\end{align*}
and similarly for each $i = 1, \dots, p-1$
\begin{equation}
\label{eq:w_betadiff_bound_bfl}
(\beta_{0,j+1} - \beta_{0,j})^2 \leq \ds \sum_{i=1}^{p-1}(\beta_{0,i+1} - \beta_{0,i})^2 \leq \left(\ds \sum_{i=1}^{p-1} w_{0,i}^2 \right) \left(\ds \sum_{i=1}^{p-1} \dfrac{(\beta_{0,i} - \beta_{0,i} )^2}{w_{0,i}^2}  \right)\leq  \dfrac{4d^2}{\lambda_2^2} := d_2^2\,.
\end{equation}
With these bounds involving $\beta_0$ and using Lemma \ref{lem:inv_gauss_minor},
\begin{align*}
f(\tau^2, w^2 \mid \beta_0, \sigma^2, y)& =  f(\tau^2 \mid \beta_0, \sigma^2, y) \, f(w^2 \mid \beta_0, \sigma^2, y)\\
& \geq \prod_{i=1}^{p}\exp\left\{ - \sqrt{\dfrac{\lambda_1^2 d_1^2}{\sigma^2}} \right\} q_i(\tau^2_i \mid \sigma^2)\;  \prod_{i=1}^{p-1}\exp\left\{ - \sqrt{\dfrac{\lambda_2^2 d_2^2}{\sigma^2}} \right\} h_i(w^2_i \mid \sigma^2)\\
& = \exp\left\{-p \sqrt{\dfrac{\lambda_1^2 d_1^2}{\sigma^2} } -p \sqrt{\dfrac{\lambda_2^2 d_2^2}{\sigma^2} }   \right\} \left[ \prod_{i=1}^{p} q_i (\tau^2_i \mid \sigma^2)\right] \left[\prod_{i=1}^{p-1} h_{i} (w^2_i \mid \sigma^2) \right]\,.
\intertext{Since for $a,b \geq 0, 2ab \leq a^2 + b^2$,}
 f(\tau^2, w^2 \mid \beta_0, \sigma^2, y) & \geq \exp \left\{- 1 - \dfrac{p^2\lambda_2^2 d_2^2}{2 \sigma^2} - \dfrac{p^2\lambda_1^2 d_1^2}{2 \sigma^2}   \right\} \left[ \prod_{i=1}^{p} q_i (\tau^2_i \mid \sigma^2)\right] \left[\prod_{i=1}^{p-1} h_{i} (w^2_i \mid \sigma^2) \right] \numberthis \label{eq:bfl_tau_gamma_minor} \,,
\end{align*}
where $q_i$ and $h_i$ are densities of the reciprocal of an Inverse-Gaussian distribution  with parameters $\sqrt{\lambda_1^2 \sigma^2/{d_1^2} }$ and $\lambda_1^2$, and $\sqrt{\lambda_2^2 \sigma^2/d_2^2}$ and $\lambda_2^2$, respectively. 

Recall the decomposition $\Sigma^{-1}_{\tau_0, w_0} = L_{0,1} + L_{0,2}$ in \eqref{eq:sigma_split}; here the 0 in the index indicates $\tau^2_{0}$ and $w^2_{0}$ entries. Here $L_{0,1}$ is the diagonal matrix with entries $1/(2\tau^2_{0,i})$.  Then since 
\begin{align*}
& y^TX(X^TX + L_{0,1} + L_{0,2})X^Ty \geq y^TX (X^TX + L_{0,1})X^Ty\\
\Rightarrow \;& y^TX(X^TX + L_{0,1} + L_{0,2})^{-1}X^Ty \leq y^TX (X^TX + L_{0,1})^{-1}X^Ty\,.
\end{align*}
Using the above, the fact that for each $i = 1, \dots, p$, $2\tau^2_{0,i} \leq 8d/\lambda_1^2 $, and Lemma \ref{lem:yty_bound_minor},
\begin{align*}
(y - X\beta_0)^T(y - X\beta_0) + \beta^T_0 \Sigma^{-1}_{\tau_0, w_0}\beta_0 & \geq y^Ty -  y^TX \left(X^TX + \Sigma^{-1}_{\tau, w} \right)^{-1}X^Ty \\
& \geq y^Ty - y^TX (X^TX + L_{0,1})^{-1}X^Ty\\
& \geq y^Ty - y^TX \left(X^TX + \dfrac{\lambda_1^2}{8d}I_p\right)^{-1}X^Ty \numberthis \label{eq:bfl_sigma_yty_bound}\,.
\end{align*}
Using  \eqref{eq:bfl_sigma_yty_bound} and the fact that for $(\beta_0, \tau^2_0, w^2_0, \sigma^2_0) \in C_d$, $(y - X\beta_0)^T(y - X\beta_0) + \beta_0^T \Sigma^{-1}_{\tau_0, w_0} \beta_0 \leq d$,
\begin{align*}
& \exp \left\{-\dfrac{1}{2} - \dfrac{p^2\lambda_2^2 d_2^2}{2 \sigma^2} -\dfrac{1}{2} - \dfrac{p^2\lambda_1^2 d_1^2}{2 \sigma^2}   \right\} f(\sigma^2 \mid \beta_0, \tau^2_0, w^2_0, y) \\
& =  \exp \left\{-1 - \dfrac{p^2\lambda_2^2 d_2^2}{2 \sigma^2} - \dfrac{p^2\lambda_1^2 d_1^2}{2 \sigma^2}   \right\}  \dfrac{\left(\frac{(y-X\beta_0)^T(y-X\beta_0) + \beta^T_0 \Sigma_{\tau_0,w_0}^{-1} \beta_0+ 2\xi }{2} \right)^{\frac{n+p}{2} + \alpha}}{\Gamma\left( \frac{n+p}{2} + \alpha  \right)}\left(\sigma^2 \right)^{-\frac{n+p}{2} - \alpha - 1}\\
& \quad \times  \exp \left\{ -\dfrac{(y-X\beta_0)^T(y-X\beta_0) + \beta^T_0 \Sigma_{\tau_0, w_0}^{-1} \beta_0 + 2\xi}{2 \sigma^2} \right\}\\
& \geq e^{-1} {\left(\frac{y^Ty -  y^TX(X^TX + \lambda_1^2(8d)^{-1}I_p)^{-1}X^Ty+ 2\xi }{2} \right)^{\frac{n+p}{2} + \alpha}} \dfrac{1}{\Gamma\left( \frac{n+p}{2} + \alpha  \right)}\left(\sigma^2 \right)^{-\frac{n+p}{2} - \alpha - 1}\\
& \quad \times \exp \left\{ -\dfrac{d+ 2\xi + p^2\lambda_2^2 d_2^2 + p^2\lambda^2_1d_1^2}{2 \sigma^2} \right\}\\
& = e^{-1}\left(\frac{y^Ty -  y^TX(X^TX + \lambda_1^2(8d)^{-1}I_p)^{-1}X^Ty+ 2\xi }{d + 2\xi + p^2\lambda_2^2 d_2^2 + p^2\lambda^2_1d_1^2} \right)^{\frac{n+p}{2} + \alpha} q(\sigma^2)\,, \numberthis \label{eq:bfl_minor_sig}
\end{align*}
where $q(\sigma^2)$ is the Inverse-Gamma density with parameters, $(n+p)/2 + \alpha$ and $d + 2\xi + p^2\lambda_2^2 d_2^2 + p^2\lambda^2_1d_1^2$. Finally, using \eqref{eq:bfl_tau_gamma_minor} and \eqref{eq:bfl_minor_sig},
\begin{equation*}
	k_{BFL}(\beta, \tau^2, w^2, \sigma^2 \mid \beta_0, \tau^2_0, w_0^2, \sigma^2_0) \geq \epsilon \; f(\beta \mid \tau^2, w^2, \sigma^2, y) \; q(\sigma^2) \left[\prod_{i=1}^{p} q_i(\tau^2_i \mid \sigma^2)\right] \left[\prod_{i=1}^{p-1} h_i (w_{i}^2 \mid \sigma^2) \right]\,, 
\end{equation*}
where
\begin{equation*}
\label{eq:epsilon}
\epsilon = e^{-1}\left(\frac{y^Ty -  y^TX(X^TX + \lambda_1(8d)^{-1}I_p)^{-1}X^Ty+ 2\xi }{d + 2\xi + p^2\lambda_2^2 d_2^2 + p^2\lambda^2_1d_1^2} \right)^{\frac{n+p}{2} + \alpha}\,.
\end{equation*}

\subsection{Starting Values} 
\label{sub:bfl_starting_values}
Starting value $(\beta_0, \tau^2_0, w^2_0, \sigma^2_0)$ can be chosen so that 
$(\beta_0, \tau^2_0, w^2_0, \sigma^2_0) = \arg \min V_{BFL}(\beta, \tau^2, w^2, \sigma^2).$
We will find the minimum by profiling out $\tau^2$ and $w^2$.
By \eqref{eq:fused_beta_expansion} in Appendix \ref{sub:validity_of_the_prior},
\begin{align*}
	\dfrac{\partial V_{BFL}}{\partial \tau^2_{0,i}} = 0 \Rightarrow & = -\dfrac{\beta^2_{0,i}}{\tau^4_{0,i}} + \dfrac{\lambda_1^2}{4} = 0 \Rightarrow \tau^2_{0,i}  = \sqrt{\dfrac{4 \beta^2_{0,i}}{\lambda_1^2} }\\
	\dfrac{\partial V_{BFL}}{\partial w^2_{0,i}} = 0 \Rightarrow & = -\dfrac{(\beta_{0,i+1} - \beta_{0,i})^2}{w^4_{0,i}} + \dfrac{\lambda_2^2}{4} = 0 \Rightarrow w^2_{0,i}  = \sqrt{\dfrac{4 (\beta_{0,i+1} - \beta_{0,i})^2}{\lambda_2^2} }\,.
\end{align*}
The $\beta_0$ that minimizes $V_{BFL}$ is,
\begin{align*}
\beta_0 & = \arg \min_{\beta \in \real^p} \Bigg\{(y - X\beta)^T(y - X\beta) + \ds \sum_{i=1}^{p} \dfrac{\lambda_1  \beta^2_i}{2 \sqrt{\beta^2_i  }}  + \ds \sum_{i=1}^{p-1} \ds \dfrac{\lambda_2 (\beta_{i+1} - \beta_{i})^2}{2 \sqrt{(\beta_{i+1} - \beta_{i})^2} }\\
& \quad + \ds \sum_{i=1}^{p} \dfrac{\lambda_1^2}{4} \sqrt{\dfrac{4 \beta^2_i }{\lambda_1^2} }  +  \dfrac{\lambda_2^2}{4}  \ds \sum_{i=1}^{p-1} \sqrt{\dfrac{4 (\beta_{i+1} - \beta_{i})^2 }{\lambda_2^2} } \Bigg \}\\ 
& = \arg \min_{\beta \in \real^p} \; \left \{ (y - X\beta)^T(y - X\beta) + \lambda_1 \ds \sum_{i=1}^{p} | \beta_i| + \lambda_2 \ds \sum_{i=1}^{p-1} |\beta_{i+1} - \beta_{i} | \right\}\,,
\end{align*}
which equivalent to the fused lasso solution. Thus, a reasonable starting value is $\beta_0$ being the fused lasso estimate, $\tau^2_{0,i} = 2|\beta_{0,i}|/\lambda_1$ and $w^2_{0,i} = 2|\beta_{0, i+1} - \beta_{0,i}|/\lambda_2$.



\section{Proof of Geometric Ergodicity in the Bayesian Group Lasso}
\label{sec:bgl_proof}

\subsection{Drift Condition}
Consider the drift function
\begin{equation}
\label{eq:drift_function}
	V_{BGL}(\beta, \tau^2, \sigma^2) = (y - X\beta)^T(y-X\beta) + \beta^T D_{\tau}^{-1} \beta + \dfrac{\lambda^2}{4}\ds \sum_{k=1}^{K} \tau^2_k\,.
\end{equation}
For the drift condition we need to show that there exists a $0 < \phi_{BGL} < 1$ and $L_{BGL} > 0$ such that,
\[\E_{(k)}\left[ V_{BGL}(\beta, \tau^2, \sigma^2) \, | \, \beta_0, \tau^2_0, \sigma^2_0   \right] \leq \phi_{BGL} V_{BGL}(\beta_0, \tau^2_0, \sigma^2_0) + L_{BGL}\,, \]
for every $(\beta_0, \tau^2_0, \sigma^2_0) \in \real^p \times \real_+^K \times \real_+$. Just as in the proof for BFL,
\begin{align*}
	\E_{(k)}\left[ V_{BGL}(\beta, \tau^2, \sigma^2) \mid \beta_0, \tau^2_0, \sigma^2_0   \right] =  \E_{\sigma^2} \left[  \E_{\tau^2} \left[  \E_{\beta} \left[V_{BGL}(\beta, \tau^2, \sigma^2)  \mid \tau^2, \sigma^2, y \right]  \mid \beta_0, \sigma^2, y \right] \mid \beta_0, \tau^2_0, y  \right].
\end{align*}
We will evaluate the expectations sequentially, starting with the innermost expectation. By Lemma \ref{lem:beta_trace_bound} and following the steps as before \eqref{eq:bfl_till_tau_1}, 
\begin{align*}
\E_{\tau^2} \left[  \E_{\beta} \left[V_{BGL}(\beta, \tau^2, \sigma^2)  \mid \tau^2, \sigma^2, y \right] \mid \beta_0, \sigma^2, y \right] & \leq  y^Ty + p \sigma^2 + \dfrac{\lambda^2}{4} \sum_{k=1}^{K} \left[ \sqrt{ \dfrac{\beta^T_{0,G_k} \beta_{0,G_k} }{\lambda^2 \sigma^2} } + \dfrac{1}{\lambda^2}  \right]\,.
\end{align*}
Let $M = \max \{ m_1,  \dots, m_K\}$. Then,
\begin{align*}
& \E_{\tau^2} \left[  \E_{\beta} \left[V_{BGL}(\beta, \tau^2, \sigma^2)  \, \mid \, \tau^2, \sigma^2, y \right]  \, \mid \, \beta_0, \sigma^2, y \right] \\
& \leq y^Ty + p \sigma^2 + \dfrac{\lambda^2}{4} \sum_{k=1}^{K} \left[ \dfrac{\beta^T_{0,G_k} \beta_{0,G_k} }{2 \sigma^2 M(n + p +2\alpha)}  + \dfrac{M(n+p+2\alpha)}{2 \lambda^2} + \dfrac{1}{\lambda^2} \right]\\
& \leq y^Ty + p \sigma^2 + \dfrac{p}{4} \left(1 + \dfrac{M(n + p + 2\alpha)}{2}  \right) + \dfrac{\lambda^2 \sum_{k=1}^{K} \beta^T_{0,G_k} \beta_{0,G_k} }{8 \sigma^2 M(n + p + 2\alpha)}\,.
\end{align*}
For the last expectation, using steps as before \eqref{eq:bfl_till_sigma_1}, we get
\begin{align*}
& \E_{\sigma^2} \left[  \E_{\tau^2} \left[  \E_{\beta} \left[V_{BGL}(\beta, \tau^2, \sigma^2)  \, \mid \, \tau^2, \sigma^2, y \right]  \, \mid \, \beta_0, \sigma^2, y \right] \, \mid \, \beta_0, \tau^2_0, y  \right]\\
& \leq y^Ty + \dfrac{p}{4} \left(1 + \dfrac{M(n + p + 2\alpha)}{2}  \right) +  \; \dfrac{\lambda^2}{8M} \left(\dfrac{\sum_{k=1}^{K} \beta^T_{0,G_k} \beta_{0,G_k}}{\beta_0^T D_{\tau_0}^{-1} \beta_0 } \right) + p \dfrac{\|y - X\beta_0\|^2 + \beta_0^T D_{\tau_0}^{-1}\beta_0 + 2\xi }{n + p + 2\alpha - 2} \,.
\end{align*}
Recall that,
\[D_{\tau_0}  = \text{diag}( \;\underbrace{\tau^2_{0,1}, \dots, \tau^2_{0,1}}_{m_1}, \underbrace{\tau^2_{0,2}, \dots, \tau^2_{0,2}}_{m_2}, \dots, \underbrace{\tau^2_{0,K}, \dots, \tau^2_{0,K}}_{m_K} )\,.\]
Let the diagonals of $D_{\tau_0}$ be $\tau^2_{0,*i}$ for $i = 1, \dots, p$. Then $\beta_0^T D_{\tau_0}^{-1} \beta_0 = \sum_{i=1}^{p} \beta^2_{0,i}/\tau^2_{0,*i}$ and $\sum_{i=1}^{p} \tau^2_{0,*i}$ $\leq M \sum_{k=1}^{K} \tau^2_{0,k}$. Using this and Lemma \ref{lem:tau_fraction_bound},
\begin{align*}
& \E_{\sigma^2} \left[  \E_{\tau^2} \left[  \E_{\beta} \left[V_{BGL}(\beta, \tau^2, \sigma^2)  \, \mid \, \tau^2, \sigma^2, y \right]  \, \mid \, \beta_0, \sigma^2, y \right] \, \mid \, \beta_0, \tau^2_0, y  \right]\\
& \leq y^Ty + \dfrac{p}{4} \left(1 + \dfrac{M(n + p + 2\alpha)}{2}  \right) +  \; \dfrac{\lambda^2}{8}\ds \sum_{k=1}^{K} \tau^2_{0,k} + p \dfrac{(y - X\beta_0)^T(y - X\beta_0) + \beta_0^T D_{\tau_0}^{-1}\beta_0 + 2\xi}{n + p + 2\alpha - 2}\\
& \leq \phi_{BGL} V_{BGL}(\beta_0, \tau^2_0, \sigma^2_0) + L_{BGL}\,,
\end{align*}
where
\begin{equation}
\label{eq:gamma_value}
\phi_{BGL} = \max \left\{ \dfrac{p}{n + p + 2\alpha - 2}, \dfrac{1}{2} \right\} < 1 \text{ for } n \geq 3\, \quad \text{ and }
\end{equation}
\begin{equation}
\label{eq:L_value}
L_{BGL} = y^Ty + \dfrac{p}{4} \left(1 + \dfrac{M(n + p + 2\alpha)}{2}  \right) + \dfrac{2p \xi}{n + p + 2\alpha - 2}\,.
\end{equation}

\subsection{Minorization Condition}

For $d > 0$, define $C_d = \{ (\beta, \tau^2, \sigma^2) : V_{BGL}(\beta, \tau^2, \sigma^2) \leq d\}$.
To establish the minorization condition, we recall that,
\begin{align*}
k_{BGL}(\beta, \tau^2, \sigma^2 \mid \beta_0, \tau^2_0, \sigma^2_0) = f(\beta \mid \tau^2, \sigma^2, y) \;f(\tau^2 \mid \beta_0, \sigma^2, y) \; f(\sigma^2 \mid \beta_0, \tau^2_0, y)\,. \numberthis \label{bgl:mtd_again}
\end{align*}
By our choice of drift function, for all $(\beta_0, \tau^2_0, \sigma_0^2) \in C_d$ the following relation holds,
\begin{align*}
(y - X\beta_0)^T(y - X\beta_0) + \beta_0^T D_{\tau_0}^{-1} \beta_0 + \dfrac{\lambda^2}{4}\ds \sum_{k=1}^{K} \tau^2_{0,k} \leq d\,. \numberthis \label{eq:bgl_drift_bound_1}
\end{align*}
By \eqref{eq:bgl_drift_bound_1}, each of $\beta_0^T D_{\tau_0}^{-1} \beta_0$ and $(\lambda^2/4)\ds \sum_{k=1}^{K} \tau^2_{0,k}$ is less than or equal to $d$, so $ \beta^T_{0, G_k} \beta_{0, G_k} \leq 4d^2/\lambda^2 := d_1^2$ for all $k = 1, \dots, K$.
By Lemma \ref{lem:inv_gauss_minor},
\begin{align*}
f(\tau^2 \mid \beta_0, \sigma^2, y) & \geq \exp\left\{- \dfrac{1}{2} - \dfrac{K^2 \lambda^2 d_1^2}{2\sigma^2} \right\} \prod_{k=1}^{K} q_k(\tau^2_k \mid \sigma^2)\,, \numberthis \label{eq:minor_tau}
\end{align*}
where $q_k$ is the density of the reciprocal of an Inverse-Gaussian distribution  with parameters $\sqrt{\lambda^2 \sigma^2/{d_1^2} }$ and $\lambda^2$. Now, since for each $i = 1, \dots, p$, $\tau^2_{0,i} \leq 4d/\lambda^2$, by Lemma \ref{lem:yty_bound_minor}
\begin{align*}
(y - X\beta_0)^T(y - X\beta_0) + \beta^T_0 D_{\tau_0}^{-1}\beta_0 
& \geq y^Ty -  y^TX \left(X^TX + \dfrac{\lambda^2}{4d}I_p \right)^{-1}X^Ty \,. \numberthis \label{eq:bgl_minor_till_tau}
\end{align*}

Using \eqref{eq:bgl_minor_till_tau} and following steps as before \eqref{eq:bfl_minor_sig}, we arrive at the following,
\begin{align*}
& \exp\left\{- \dfrac{1}{2} - \dfrac{K^2 \lambda^2 d_1^2}{2\sigma^2} \right\}  f(\sigma^2 \mid \beta_0, \tau^2_0, y) \\
& \geq e^{-\frac{1}{2}}\left(\frac{y^Ty -  y^TX(X^TX + \lambda^2(4d)^{-1}I_p)^{-1}X^Ty+ 2\xi }{d + 2\xi + K^2\lambda^2 d_1^2} \right)^{\frac{n+p}{2} + \alpha} q(\sigma^2)\,, \numberthis \label{eq:minor_sig}
\end{align*}
where $q(\sigma^2)$ is the Inverse-Gamma density with parameters, $(n+p)/2 + \alpha$ and $d + 2\xi + K^2 \lambda^2 d_1^2$.

Finally, using \eqref{eq:minor_tau} and \eqref{eq:minor_sig} in \eqref{bgl:mtd_again}
\begin{equation}
\label{eq:minor}
	k_{BGL}(\beta, \tau^2, \sigma^2 \mid \beta_0, \tau^2_0, \sigma^2_0)  \geq \; \epsilon \; f(\beta \mid \tau^2, \sigma^2, y) \; q(\sigma^2) \prod_{k=1}^{K} q_k(\tau^2 \mid \sigma^2)\,,
\end{equation}

where
\begin{equation}
\label{eq:epsilon}
	\epsilon = e^{-\frac{1}{2}} \left(\frac{y^Ty -  y^TX(X^TX + \lambda^2(4d)^{-1}I_p)^{-1}X^Ty+ 2\xi }{d + 2\xi + 4K^2d^2} \right)^{\frac{n+p}{2} + \alpha}\,.
\end{equation}

\subsection{Starting Values} 
\label{sub:bgl_starting_values}
As before, we first differentiate with respect to $\tau^2$ and then with respect to $\beta$.  Note that
\begin{align*}
	\dfrac{\partial V_{BGL}}{\partial \tau^2_{0,k}} = 0 \Rightarrow & = -\dfrac{\beta^T_{0,G_k} \beta_{0,G_k}}{\tau^4_{0,k}} + \dfrac{\lambda^2}{4} = 0 \Rightarrow \tau^2_{0,k}  = \sqrt{\dfrac{4 \beta^T_{0,G_k}\beta_{0,G_k}}{\lambda^2} }\,.
\end{align*}
Thus, the $\beta_0$ that minimizes $V_{BGL}$ is then,
\begin{align*}
\beta_0 & = \arg \min_{\beta \in \real^p}\;\; (y - X\beta)^T(y - X\beta) + \ds \sum_{k=1}^{K} \dfrac{\lambda \; \beta^T_{G_k}\beta_{G_k}}{2 \sqrt{\beta^T_{G_k}\beta_{G_k}  }} + \dfrac{\lambda^2}{4} \ds \sum_{k=1}^{K} \sqrt{ \dfrac{4 \beta^T_{G_k}\beta_{G_k}}{\lambda^2 }} \\
& = \arg \min_{\beta \in \real^p} \; \; (y - X\beta)^T(y - X\beta) + \lambda \ds \sum_{k=1}^{K} \sqrt{ \beta^T_{G_k} \beta_{G_k} }\,,
\end{align*}
which equivalent to the group lasso solution. Thus a reasonable starting value for the Markov chain is $\beta_0$ being the group lasso estimate and $\tau^2_{0,k} = 2\sqrt{\beta^T_{0,G_k} \beta_{0, G_k}}/\lambda$.

\section{Proof of Geometric Ergodicity in the Bayesian Sparse Group Lasso} 
\label{sec:bsgl_proof}

\subsection{Drift Condition} 
\label{sub:bsgl_drift_condition}

Consider the drift function 
\begin{equation}
\label{eq:drift_function_2}
	V_{BSGL}(\beta, \tau^2, \gamma^2, \sigma^2) = (y - X\beta)^T(y-X\beta) + \beta^T V_{\tau, \gamma}^{-1} \beta + \dfrac{\lambda_1^2}{4}\ds \sum_{k=1}^{K} \tau^2_k + \dfrac{\lambda_2^2}{4}\ds \sum_{k=1}^{K} \sum_{j=1}^{m_k} \gamma^2_{k,j}\,.
\end{equation}
%
By Lemma \ref{lem:beta_trace_bound} and following the steps as before \eqref{eq:bfl_till_tau_1}
\begin{align*}
& \E_{\tau^2, \gamma^2} \left[  \E_{\beta} \left[V_{BSGL}(\beta, \tau^2, \gamma^2, \sigma^2)  \, \mid \, \tau^2, \gamma^2, \sigma^2, y \right]  \, \mid \, \beta_0, \sigma^2, y \right] \\
& \leq y^Ty + p \sigma^2 + \dfrac{\lambda_1^2}{4} \sum_{k=1}^{K} \left[ \sqrt{ \dfrac{\beta^T_{0,G_k} \beta_{0,G_k} }{\lambda_1^2 \sigma^2} } + \dfrac{1}{\lambda_1^2}  \right] + \dfrac{\lambda^2_2}{4} \ds \sum_{k=1}^{K}\sum_{j=1}^{m_k} \left[ \sqrt{\dfrac{\beta_{0,k,j}^2}{\lambda_2^2 \sigma^2} } + \dfrac{1}{\lambda^2_2} \right]\,.
\end{align*}
Define $M = \max \{ m_1,  \dots, m_K\}$. In addition, define 
\[A = \left(1 + \dfrac{\lambda_1^2}{\lambda_2^2} + \dfrac{\lambda_2^2}{\lambda_1^2}\right)(n+p+ 2 \alpha)\,.\]
Then,
\begin{align*}
& \E_{\tau^2, \gamma^2} \left[  \E_{\beta} \left[V_{BSGL}(\beta, \tau^2, \sigma^2)  \, \mid \, \tau^2, \sigma^2, y \right]  \, \mid \, \beta_0, \sigma^2, y \right] \\
& \leq y^Ty + p \sigma^2 + \dfrac{\lambda_1^2}{4} \sum_{k=1}^{K} \left[ \dfrac{\beta^T_{0,G_k} \beta_{0,G_k} }{2 \sigma^2 AM}  + \dfrac{AM}{2 \lambda_1^2} + \dfrac{1}{\lambda_1^2} \right] + \dfrac{\lambda^2_2}{4} \ds \sum_{k=1}^{K} \sum_{j=1}^{m_k} \left[ \dfrac{\beta_{0,k,j}^2}{2 \sigma^2 AM} + \dfrac{AM}{2 \lambda^2_2} + \dfrac{1}{\lambda^2_2} \right]\\
& = y^Ty + p \sigma^2 + \dfrac{p}{4} \left(2 + AM  \right) +  \left[\dfrac{\lambda_1^2 + \lambda_2^2 }{8AM} \right] \dfrac{\beta^T_0\beta_0}{\sigma^2}\,.
\end{align*}
For the last expectation, using steps as before \eqref{eq:bfl_till_sigma_1}, we get
\begin{align*}
& \E_{\sigma^2} \left[  \E_{\tau^2, \gamma^2} \left[  \E_{\beta} \left[V_{BSGL}(\beta, \tau^2, \gamma^2,  \sigma^2)  \, \mid \, \tau^2, \gamma^2, \sigma^2, y \right]  \, \mid \, \beta_0, \sigma^2, y \right] \, \mid \, \beta_0, \tau^2_0, \gamma^2_0, y  \right]\\
& \leq y^Ty + \dfrac{p}{4} \left(2 + AM  \right) +  \; (\lambda_1^2 + \lambda_2^2) \left[8M \left(1 + \dfrac{\lambda_1^2}{\lambda_2^2} + \dfrac{\lambda_2^2}{\lambda_1^2}\right) \right]^{-1} \left(\dfrac{ \beta^T_{0} \beta_{0}}{\beta_0^T V_{\tau_0, \gamma_0}^{-1} \beta_0 } \right)\\
& \quad + p \dfrac{(y - X\beta_0)^T(y - X\beta_0) + \beta_0^T V_{\tau_0, \gamma_0}^{-1}\beta_0 + 2\xi }{n + p + 2\alpha - 2} \,. \numberthis \label{eq:bsgl_till_sigma_1}
\end{align*}
Let $v_{0,i}$ denote the diagonals of $V_{\tau_0, \gamma_0}$. Then by Lemma \ref{lem:tau_fraction_bound},
and the fact that the harmonic mean of positive numbers is less than their arithmetic mean,
\begin{align*}
\dfrac{\beta^T_{0} \beta_{0}}{\beta_0^T V_{\tau_0, \gamma_0}^{-1} \beta_0 } & \leq \ds \sum_{i=1}^{p} v_{0,i} = \ds \sum_{k=1}^{K} \sum_{j= 1}^{m_k} \left( \dfrac{1}{\tau^2_{0,k}} + \dfrac{1}{\gamma_{0,k,j}^2}\right)^{-1} = \dfrac{1}{2} \ds \sum_{k=1}^{K} \sum_{j= 1}^{m_k} 2\left( \dfrac{1}{\tau^2_{0,k}} + \dfrac{1}{\gamma_{0,k,j}^2}\right)^{-1}\\
& \leq \dfrac{1}{2} \ds \sum_{k=1}^{K} \sum_{j= 1}^{m_k} \dfrac{\tau^2_{0,k} + \gamma^2_{0,k,j} }{2} \leq \dfrac{M}{4} \sum_{k=1}^{K} \tau^2_{0,k} + \dfrac{1}{4} \sum_{k=1}^{K} \sum_{j=1}^{m_k} \gamma^2_{0,k,j}\,. \numberthis \label{eq:bsgl_diag_bound}
\end{align*}
Using \eqref{eq:bsgl_diag_bound} in \eqref{eq:bsgl_till_sigma_1},
\begin{align*}
& \E_{\sigma^2} \left[  \E_{\tau^2} \left[  \E_{\beta} \left[V_{BSGL}(\beta, \tau^2, \sigma^2)  \, \mid \, \tau^2, \sigma^2, y \right]  \, \mid \, \beta_0, \sigma^2, y \right] \, \mid \, \beta_0, \tau^2_0, y  \right]\\
& \leq y^Ty + \dfrac{p}{4} \left(2 + AM  \right) +  \; (\lambda_1^2 + \lambda_2^2)\left[8M \left(1 + \dfrac{\lambda_1^2}{\lambda_2^2} + \dfrac{\lambda_2^2}{\lambda_1^2}\right) \right]^{-1} \left(\dfrac{M}{4} \sum_{k=1}^{K} \tau^2_{0,k} + \dfrac{1}{4} \sum_{k=1}^{K} \sum_{j=1}^{m_k} \gamma^2_{0,k,j} \right)\\
& \quad + p \dfrac{(y - X\beta_0)^T(y - X\beta_0) + \beta_0^T V_{\tau_0, \gamma_0}^{-1}\beta_0 + 2\xi }{n + p + 2\alpha - 2}\\
& \leq y^Ty + \dfrac{p}{4} \left(2 + AM  \right) + \dfrac{2 p \xi }{n + p + 2 \alpha  - 2} + \dfrac{p}{n + p + 2\alpha - 2} \left[(y - X\beta_0)^T(y - X\beta_0) + \beta_0^T V_{\tau_0, \gamma_0}^{-1}\beta_0 \right]\\
& \quad + \left(1 + \dfrac{\lambda_2^2}{\lambda_1^2}\right) \left[8 \left(1 + \dfrac{\lambda_1^2}{\lambda_2^2} + \dfrac{\lambda_2^2}{\lambda_1^2}\right) \right]^{-1}  \left(\dfrac{\lambda_1^2}{4}\ds \sum_{k=1}^{K} \tau^2_{0,k} \right)\\
& \quad +  \left(1 + \dfrac{\lambda_1^2}{\lambda_2^2}\right) \left[8 M \left(1 + \dfrac{\lambda_1^2}{\lambda_2^2} + \dfrac{\lambda_2^2}{\lambda_1^2}\right) \right]^{-1}  \left(\dfrac{\lambda_2^2}{4}\ds \sum_{k=1}^{K} \sum_{j=1}^{m_k} \gamma^2_{0,k,j} \right)\\
& \leq \phi_{BSGL} \; V_{BSGL}(\beta_0, \tau^2_0, \gamma^2_0, \sigma^2_0) + L_{BSGL}\,,
\end{align*}
where
\begin{equation}
\label{eq:phi_bsgl}
\phi_{BSGL} = \max \left\{\dfrac{p}{n + p + 2\alpha - 2}, \dfrac{\left(1 + \dfrac{\lambda_2^2}{\lambda_1^2} \right)}{8 \left(1 + \dfrac{\lambda_1^2}{\lambda_2^2} + \dfrac{\lambda_2^2}{\lambda_1^2}  \right)}, \dfrac{\left(1 + \dfrac{\lambda_1^2}{\lambda_2^2} \right)}{8 M \left(1 + \dfrac{\lambda_1^2}{\lambda_2^2} + \dfrac{\lambda_2^2}{\lambda_1^2}  \right)}    \right\} < 1 \text{ for } n\geq 3\,,
\end{equation}
and
\begin{equation}
\label{eq:L_bsgl}
L_{BSGL} = y^Ty + \dfrac{p}{4} \left(2 + AM  \right) + \dfrac{2 p \xi }{n + p + 2 \alpha  - 2}\,.
\end{equation}

\subsection{Minorization} 
\label{sub:minorization}

For $d > 0$, define $C_d = \{ (\beta, \tau^2, \gamma^2, \sigma^2) : V(\beta, \tau^2, \gamma^2, \sigma^2) \leq d\}$. Recall that,
\begin{align*}
& k_{BSGL}(\beta, \tau^2, \gamma^2, \sigma^2 | \beta_0, \tau^2_0, \gamma^2_0, \sigma^2_0) = f(\beta | \tau^2, \gamma^2, \sigma^2, y) \;f(\tau^2, \gamma^2 | \beta_0, \sigma^2, y) \; f(\sigma^2 |\beta_0, \tau^2_0, \gamma_0^2, y)\,. \numberthis \label{eq:bsgl_mtd_again}
\end{align*}
By our definition of the drift function, for all $(\beta_0, \tau^2_0, \gamma_0^2, \sigma_0^2) \in C_d$ the following relation holds:
\begin{align*}
(y - X\beta_0)^T(y - X\beta_0) + \ds \sum_{k=1}^{K} \dfrac{ \beta_{0,G_k}^T \beta_{0, G_k}}{\tau^2_{0,k} } + \ds \sum_{k=1}^{K} \sum_{j=1}^{m_k} \dfrac{\beta^2_{0,k,j}}{\gamma^2_{0,k,j}}  + \dfrac{\lambda_1^2}{4}\ds \sum_{k=1}^{K} \tau^2_{0,k}  + \dfrac{\lambda_2^2}{4} \ds \sum_{k=1}^{K} \sum_{j=1}^{m_k} \gamma^2_{0,k,j} & \leq d\,.
\end{align*}
Using the above and following on the lines of \eqref{eq:bgl_drift_bound_1} we get for all $k = 1, \dots, K$ and $j = 1, \dots, m_k$
\begin{align*}
	 & \beta^T_{0, G_k} \beta_{0, G_k} \leq \dfrac{4d^2}{\lambda_1^2} := d_1^2 \quad \text{ and } \quad \beta_{0,k,j}^2 \leq \dfrac{4d^2}{\lambda_2^2} := d_2^2\,. \numberthis  \label{eq:beta_gamma_bound}
\end{align*}


Using Lemma \ref{lem:inv_gauss_minor} and \eqref{eq:beta_gamma_bound} and following steps as before \eqref{eq:bfl_tau_gamma_minor},
\begin{align*}
f(\tau^2, \gamma^2 \mid \beta_0, \sigma^2, y)
& \geq \exp \left\{-1 - \dfrac{p^2\lambda_2^2 d_2^2}{2 \sigma^2} - \dfrac{K^2\lambda_1^2 d_1^2}{2 \sigma^2}   \right\} \prod_{k=1}^{K} \left[q_k (\tau^2_k \mid \sigma^2) \prod_{j=1}^{m_k} q_{k,j} (\gamma^2_{k,j} \mid \sigma^2) \right]\,, \numberthis \label{eq:bsgl_tau_gamma_minor}
\end{align*}
where $q_k(\tau^2_k \mid \sigma^2)$ and  $q_{k,j}(\gamma^2_{k,j} \mid \sigma^2)$ are the densities of the reciprocal of an Inverse-Gaussian distribution  with parameters $\sqrt{\lambda_1^2 \sigma^2/{d_1^2} }$ and $\lambda_1^2$, and  $\sqrt{\lambda_2^2 \sigma^2/d_2^2}$ and $\lambda_2^2$, respectively.  Since each $\tau_{0,k}^2 \leq 4d/\lambda_1^2$ and each $\gamma_{0,k,j}^2 \leq 4d/\lambda_2^2$, so
\[\left( \dfrac{1}{\tau_{0,k}^2} + \dfrac{1}{\gamma_{0,k,j}^2} \right)^{-1} \leq \left( \dfrac{\lambda_1^2}{4d} + \dfrac{\lambda^2_2}{4d}  \right)^{-1} := d_3\,. \]

By Lemma \ref{lem:yty_bound_minor}
\begin{align*}
(y - X\beta_0)^T(y - X\beta_0) + \beta^T_0 V_{\tau_0, \gamma_0}^{-1}\beta_0 
& \geq y^Ty -  y^TX \left(X^TX + \dfrac{1}{d_3}I_p \right)^{-1}X^Ty \numberthis \label{eq:sigma_yty_bound}\,.
\end{align*}
Using \eqref{eq:sigma_yty_bound} and following steps as before \eqref{eq:bfl_minor_sig}
\begin{align*}
& \exp \left\{-1 - \dfrac{p^2\lambda^2_2 d_2^2}{2 \sigma^2} - \dfrac{K^2\lambda^2_1 d_1^2}{2 \sigma^2}   \right\} f(\sigma^2 \mid \beta_0, \tau^2_0, \gamma^2_0, y) \\
& = e^{-1}\left(\frac{y^Ty -  y^TX(X^TX + d_3^{-1}I_p)^{-1}X^Ty+ 2\xi }{d + 2\xi + p^2\lambda_2^2 d_2^2 + K^2\lambda^2_1d_1^2} \right)^{\frac{n+p}{2} + \alpha} q(\sigma^2)\,, \numberthis \label{eq:minor_sig_bsgl}
\end{align*}
where $q(\sigma^2)$ is the density of the Inverse-Gamma distribution with parameters, $(n+p)/2 + \alpha$ and $d + 2\xi + p^2\lambda_2^2 d_2^2 + K^2\lambda^2_1d_1^2$. Using \eqref{eq:bsgl_tau_gamma_minor} and \eqref{eq:minor_sig_bsgl} in \eqref{eq:bsgl_mtd_again},
\begin{align*}
\label{eq:minor}
	& k_{BSGL}(\beta, \tau^2, \gamma^2, \sigma^2 \mid \beta_0, \tau^2_0, \gamma_0^2, \sigma^2_0) \geq \; \epsilon \; f(\beta \mid \tau^2, \gamma^2, \sigma^2, y) \; q(\sigma^2) \prod_{k=1}^{K} \left[q_k(\tau^2_k \mid \sigma^2) \prod_{j=1}^{m_k} q_{k,j} (\gamma_{k,j}^2 \mid \sigma^2) \right]\,,
\end{align*}
where
\begin{equation}
\label{eq:epsilon}
\epsilon = e^{-1}\left(\frac{y^Ty -  y^TX(X^TX + d_3^{-1}I_p)^{-1}X^Ty+ 2\xi }{d + 2\xi + p^2\lambda_2^2 d_2^2 + K^2\lambda^2_1d_1^2} \right)^{\frac{n+p}{2} + \alpha}\,.
\end{equation}


\subsection{Starting Values} 
\label{sub:bsgl_starting_values}
To minimize $V_{BSGL}$,
\begin{align*}
	\dfrac{\partial V_{BSGL}}{\partial \tau^2_{0,k}} = 0 \Rightarrow & = -\dfrac{\beta^T_{0,G_k} \beta_{0,G_k}}{\tau^4_{0,k}} + \dfrac{\lambda_1^2}{4} = 0 \Rightarrow \tau^2_{0,k}  = \sqrt{\dfrac{4 \beta^T_{0,G_k}\beta_{0,G_k}}{\lambda_1^2} }\\
	\dfrac{\partial V_{BSGL}}{\partial \gamma^2_{0,k,j}} = 0 \Rightarrow & = -\dfrac{\beta^2_{0,k,j}}{\gamma^4_{0,k,j}} + \dfrac{\lambda_2^2}{4} = 0 \Rightarrow \gamma^2_{0,k,j}  = \sqrt{\dfrac{4 \beta^2_{0,k,j}}{\lambda_2^2} }\,.
\end{align*}
For the starting value for $\beta$,
\begin{align*}
\beta_0 & = \arg \min_{\beta \in \real^p} \Bigg\{ \;\; (y - X\beta)^T(y - X\beta) + \ds \sum_{k=1}^{K} \dfrac{\lambda_1  \beta^T_{G_k}\beta_{G_k}}{2 \sqrt{\beta^T_{G_k}\beta_{G_k}  }}  + \ds \sum_{k=1}^{K} \sum_{j=1}^{m_k} \ds \dfrac{\lambda_2 \beta_{k,j}^2}{2 \sqrt{\beta_{k,j}^2} }\\
& \quad + \ds \sum_{k=1}^{K} \dfrac{\lambda_1^2}{4} \sqrt{\dfrac{4 \beta^T_{G_k}\beta_{G_k} }{\lambda_1^2} }  +  \dfrac{\lambda_2^2}{4}  \ds \sum_{k=1}^{K} \sum_{j=1}^{m_k} \sqrt{\dfrac{4 \beta_{k,j}^2 }{\lambda_2^2} } \Bigg \}\\ 
& = \arg \min_{\beta \in \real^p} \; \; (y - X\beta)^T(y - X\beta) + \lambda_1 \ds \sum_{k=1}^{K} \sqrt{ \beta^T_{G_k} \beta_{G_k} } + \lambda_2 \ds \sum_{k=1}^{K} \sum_{j=1}^{m_k} |\beta_{k,j} |\,,
\end{align*}
which corresponds to the sparse group lasso solutions. Thus a reasonable starting value for  is $\beta_0$ being the sparse group lasso estimate, $\tau^2_{0,k} = 2\sqrt{\beta^T_{0,G_k} \beta_{0,G_k} }/\lambda_1$ and $\gamma^2_{0,k} = 2 |\beta_{0,k,j} |/ \lambda_2$.


\bibliographystyle{apalike}
\bibliography{mcref}

\begin{thebibliography}{}

\bibitem[Andeli{\'c} and Da~Fonseca, 2011]{andj:dafo:2011}
Andeli{\'c}, M. and Da~Fonseca, C. (2011).
\newblock Sufficient conditions for positive definiteness of tridiagonal
  matrices revisited.
\newblock {\em Positivity}, 15:155--159.

\bibitem[Bhattacharya et~al., 2016]{bhatt:chak:mall:2015}
Bhattacharya, A., Chakraborty, A., and Mallick, B.~K. (2016).
\newblock Fast sampling with {G}aussian scale mixture priors in
  high-dimensional regression.
\newblock {\em Biometrika}, 103:985--991.

\bibitem[Doss et~al., 2014]{doss:fleg:jone:neat:2014}
Doss, C.~R., Flegal, J.~M., Jones, G.~L., and Neath, R.~C. (2014).
\newblock Markov chain {M}onte {C}arlo estimation of quantiles.
\newblock {\em Electronic Journal of Statistics}, 8:2448--2478.

\bibitem[Doss and Hobert, 2010]{doss:hobe:2010}
Doss, H. and Hobert, J.~P. (2010).
\newblock Estimation of {B}ayes factors in a class of hierarchical random
  effects models using a geometrically ergodic {MCMC} algorithm.
\newblock {\em Journal of Computational and Graphical Statistics}, 19:295--312.

\bibitem[Fan et~al., 2017]{fan:wang:peng:2017}
Fan, Y., Wang, X., and Peng, Q. (2017).
\newblock Inference of gene regulatory networks using {B}ayesian nonparametric
  regression and topology information.
\newblock {\em Computational and Mathematical Methods in Medicine}, 2017.

\bibitem[Flegal and Gong, 2015]{fleg:gong:2015}
Flegal, J.~M. and Gong, L. (2015).
\newblock Relative fixed-width stopping rules for {M}arkov chain {M}onte
  {C}arlo simulations.
\newblock {\em Statistica Sinica}, 25:655--676.

\bibitem[Flegal and Jones, 2010]{fleg:jone:2010}
Flegal, J.~M. and Jones, G.~L. (2010).
\newblock Batch means and spectral variance estimators in {M}arkov chain
  {M}onte {C}arlo.
\newblock {\em The Annals of Statistics}, 38:1034--1070.

\bibitem[Gong and Flegal, 2016]{gong:fleg:2015}
Gong, L. and Flegal, J.~M. (2016).
\newblock A practical sequential stopping rule for high-dimensional {M}arkov
  chain {M}onte {C}arlo.
\newblock {\em Journal of Computational and Graphical Statistics}, pages
  684--700.

\bibitem[Griffin and Brown, 2010]{grif:brown:2010}
Griffin, J.~E. and Brown, P.~J. (2010).
\newblock Inference with normal-gamma prior distributions in regression
  problems.
\newblock {\em Bayesian Analysis}, 5:171--188.

\bibitem[Gu et~al., 2013]{gu:yin:lee:2013}
Gu, X., Yin, G., and Lee, J.~J. (2013).
\newblock {B}ayesian two-step lasso strategy for biomarker selection in
  personalized medicine development for time-to-event endpoints.
\newblock {\em Contemporary clinical trials}, 36:642--650.

\bibitem[Guan and Stephens, 2011]{guan:step:2011}
Guan, Y. and Stephens, M. (2011).
\newblock Bayesian variable selection regression for genome-wide association
  studies and other large-scale problems.
\newblock {\em The Annals of Applied Statistics}, pages 1780--1815.

\bibitem[Hobert et~al., 2002]{hobe:jone:pres:rose:2002}
Hobert, J.~P., Jones, G.~L., Presnell, B., and Rosenthal, J.~S. (2002).
\newblock On the applicability of regenerative simulation in {M}arkov chain
  {M}onte {C}arlo.
\newblock {\em Biometrika}, 89:731--743.

\bibitem[Ishwaran and Rao, 2005]{ishw:rao:2005}
Ishwaran, H. and Rao, J.~S. (2005).
\newblock Spike and slab variable selection: frequentist and {B}ayesian
  strategies.
\newblock {\em Annals of Statistics}, pages 730--773.

\bibitem[Johnson and Jones, 2015]{john:jone:2015}
Johnson, A.~A. and Jones, G.~L. (2015).
\newblock Geometric ergodicity of random scan {G}ibbs samplers for hierarchical
  one-way random effects models.
\newblock {\em Preprint}.

\bibitem[Jones et~al., 2006]{jone:hara:caff:neat:2006}
Jones, G.~L., Haran, M., Caffo, B.~S., and Neath, R. (2006).
\newblock Fixed-width output analysis for {M}arkov chain {M}onte {C}arlo.
\newblock {\em Journal of the American Statistical Association},
  101:1537--1547.

\bibitem[Jones and Hobert, 2001]{jone:hobe:2001}
Jones, G.~L. and Hobert, J.~P. (2001).
\newblock Honest exploration of intractable probability distributions via
  {M}arkov chain {M}onte {C}arlo.
\newblock {\em Statistical Science}, 16:312--334.

\bibitem[Jones and Hobert, 2004]{jone:hobe:2004}
Jones, G.~L. and Hobert, J.~P. (2004).
\newblock Sufficient burn-in for {G}ibbs samplers for a hierarchical random
  effects model.
\newblock {\em The Annals of Statistics}, 32:784--817.

\bibitem[Khare and Hobert, 2012]{khar:hob:2012}
Khare, K. and Hobert, J.~P. (2012).
\newblock Geometric ergodicity of the {G}ibbs sampler for {B}ayesian quantile
  regression.
\newblock {\em Journal of Multivariate Analysis}, 112:108--116.

\bibitem[Khare and Hobert, 2013]{khare:hobe:2013}
Khare, K. and Hobert, J.~P. (2013).
\newblock Geometric ergodicity of the {B}ayesian lasso.
\newblock {\em Electronic Journal of Statistics}, 7:2150--2163.

\bibitem[Kyung et~al., 2010]{kyu:gill:gho:cas:2010}
Kyung, M., Gill, J., Ghosh, M., and Casella, G. (2010).
\newblock Penalized regression, standard errors, and {B}ayesian lassos.
\newblock {\em Bayesian Analysis}, 5:369--411.

\bibitem[Liu et~al., 1994]{liu:wong:kong:1994}
Liu, J.~S., Wong, W.~H., and Kong, A. (1994).
\newblock Covariance structure of the {G}ibbs sampler with applications to the
  comparisons of estimators and augmentation schemes.
\newblock {\em Biometrika}, 81:27--40.

\bibitem[Meyn and Tweedie, 2009]{meyn:twee:2009}
Meyn, S.~P. and Tweedie, R.~L. (2009).
\newblock {\em Markov Chains and Stochastic Stability}.
\newblock Cambridge University Press.

\bibitem[Nathoo et~al., 2016]{nath:green:lesp:2016}
Nathoo, F.~S., Greenlaw, K., and Lesperance, M. (2016).
\newblock Regularization parameter selection for a {B}ayesian multi-level group
  lasso regression model with application to imaging genomics.
\newblock {\em arXiv preprint arXiv:1603.08163}.

\bibitem[Pal and Khare, 2014]{pal:khar:2014}
Pal, S. and Khare, K. (2014).
\newblock Geometric ergodicity for {B}ayesian shrinkage models.
\newblock {\em Electronic Journal of Statistics}, 8:604--645.

\bibitem[Park and Casella, 2008]{park:cas:2008}
Park, T. and Casella, G. (2008).
\newblock The {B}ayesian lasso.
\newblock {\em Journal of the American Statistical Association}, 103:681--686.

\bibitem[Rajaratnam and Sparks, 2015]{raj:spark:2015}
Rajaratnam, B. and Sparks, D. (2015).
\newblock {MCMC}-based inference in the era of big data: A fundamental analysis
  of the convergence complexity of high-dimensional chains.
\newblock {\em arXiv preprint arXiv:1508.00947}.

\bibitem[Raman et~al., 2010]{ram:fuc:wild:2010}
Raman, S., Fuchs, T.~J., Wild, P.~J., Dahl, E., Buhmann, J.~M., and Roth, V.
  (2010).
\newblock Infinite mixture-of-experts model for sparse survival regression with
  application to breast cancer.
\newblock {\em BMC bioinformatics}, 11:1.

\bibitem[Roberts and Rosenthal, 1997]{robe:rose:1997c}
Roberts, G.~O. and Rosenthal, J.~S. (1997).
\newblock Geometric ergodicity and hybrid {M}arkov chains.
\newblock {\em Electronic Communications in Probability}, 2:13--25.

\bibitem[Roberts and Rosenthal, 2004]{robe:rose:2004}
Roberts, G.~O. and Rosenthal, J.~S. (2004).
\newblock General state space {M}arkov chains and {MCMC} algorithms.
\newblock {\em Probability Surveys}, 1:20--71.

\bibitem[Rosenthal, 1995]{rose:1995a}
Rosenthal, J.~S. (1995).
\newblock Minorization conditions and convergence rates for {M}arkov chain
  {M}onte {C}arlo.
\newblock {\em Journal of the American Statistical Association}, 90:558--566.

\bibitem[Roy and Chakraborty, 2017]{roy:chakra:2016}
Roy, V. and Chakraborty, S. (2017).
\newblock Selection of tuning parameters, solution paths and standard errors
  for {B}ayesian lassos.
\newblock {\em Bayesian Analysis}, 12:753--778.

\bibitem[Simon et~al., 2013]{sim:fried:has:2013}
Simon, N., Friedman, J., Hastie, T., and Tibshirani, R. (2013).
\newblock A sparse-group lasso.
\newblock {\em Journal of Computational and Graphical Statistics}, 22:231--245.

\bibitem[Tibshirani et~al., 2005]{tib:saun:ross:2005}
Tibshirani, R., Saunders, M., Rosset, S., Zhu, J., and Knight, K. (2005).
\newblock Sparsity and smoothness via the fused lasso.
\newblock {\em Journal of the Royal Statistical Society: Series B (Statistical
  Methodology)}, 67:91--108.

\bibitem[Vats et~al., 2015a]{vats:fleg:jones:2015b}
Vats, D., Flegal, J.~M., and Jones, G.~L. (2015a).
\newblock Multivariate output analysis for {M}arkov chain {M}onte {C}arlo.
\newblock {\em arXiv preprint arXiv:1512.07713}.

\bibitem[Vats et~al., 2015b]{vats:fleg:jones:2015}
Vats, D., Flegal, J.~M., and Jones, G.~L. (2015b).
\newblock Strong consistency of multivariate spectral variance estimators in
  {M}arkov chain {M}onte {C}arlo.
\newblock {\em Bernoulli \rm{(to appear})}.

\bibitem[Xu and Ghosh, 2015]{xu:gho:2015}
Xu, X. and Ghosh, M. (2015).
\newblock Bayesian variable selection and estimation for group lasso.
\newblock {\em Bayesian Analysis}, 10:909--936.

\bibitem[Yang et~al., 2016]{yang:jor:wain:2016}
Yang, Y., Wainwright, M.~J., and Jordan, M.~I. (2016).
\newblock On the computational complexity of high-dimensional {B}ayesian
  variable selection.
\newblock {\em The Annals of Statistics}, 44:2497--2532.

\bibitem[Yuan and Lin, 2006]{yuan:lin:2006}
Yuan, M. and Lin, Y. (2006).
\newblock Model selection and estimation in regression with grouped variables.
\newblock {\em Journal of the Royal Statistical Society: Series B (Statistical
  Methodology)}, 68:49--67.

\bibitem[Zhang et~al., 2014]{zhang2014bayesian}
Zhang, L., Baladandayuthapani, V., Mallick, B.~K., Manyam, G.~C., Thompson,
  P.~A., Bondy, M.~L., and Do, K.-A. (2014).
\newblock Bayesian hierarchical structured variable selection methods with
  application to molecular inversion probe studies in breast cancer.
\newblock {\em Journal of the Royal Statistical Society: Series C (Applied
  Statistics)}, 63:595--620.

\end{thebibliography}
\end{document}